\documentclass[12pt]{article}
\usepackage{amsthm,amssymb, amsmath, amscd, amsfonts,graphicx, latexsym,verbatim}
\usepackage{cite, psfrag,color}
\usepackage{titlesec}
\usepackage{epstopdf}
\usepackage{fancybox}
\usepackage{float}
\usepackage[colorlinks,citecolor=blue]{hyperref}
\usepackage[numbers,square]{natbib}
\usepackage[font=normalsize,labelsep=none]{caption}
\usepackage{diagbox}
\usepackage{changepage}
\usepackage{enumerate}
\usepackage{arydshln}
\usepackage{bbm}
\usepackage{setspace}
\usepackage{stmaryrd} 

\titleformat{\section}
{\raggedright\Large\bfseries}
{\thesection .\quad}
{0pt}
{}
\titleformat{\subsection}
{\raggedright\large\bfseries}
{\thesubsection .\quad}
{0pt}
{}

 \newtheorem{theorem}{Theorem}[section]
 \newtheorem{definition}{Definition}[section]
 \newtheorem{lemma}{Lemma}[section]
 \newtheorem{corollary}{Corollary}[section]
 \newtheorem{example}{Example}[section]
 \newtheorem{remark}{Remark}[section]

 \newtheorem{proposition}{Proposition}[section]



\newlength{\boxedparwidth}
  {\begin{center} \begin{tabular}{|@{\hspace{.315in}}c@{\hspace{.15in}}|}
                  \hline \\ \begin{minipage}[t]{\boxedparwidth}
                  \setlength{\parindent}{.25in}}%
  {\end{minipage} \\ \\ \hline \end{tabular} \end{center}}
\begin{document}
\begin{center}
{\LARGE{\bf{Mahonian and Euler-Mahonian statistics for set partitions  \vskip 1mm}}}
\end{center}
\vspace{2.0mm}
\begin{center}
\text{Shao-Hua Liu}\\
   \vskip 2.5mm
{School of Statistics and Mathematics\\
Guangdong University of Finance and Economics\\
Guangzhou, China\\

   \vskip 2.5mm
Email: liushaohua@gdufe.edu.cn}
\end{center}
 \title{}

 \vskip 4mm

\noindent {\bf Abstract.}
A partition of the set $[n]:=\{1,2,\ldots,n\}$ is a collection of disjoint nonempty subsets (or
blocks) of $[n]$, whose union is $[n]$.
In this paper we consider the following rarely used representation for set partitions:
given a partition of $[n]$ with blocks $B_{1},B_{2},\ldots,B_{m}$
satisfying $\max B_{1}<\max B_{2}<\cdots<\max B_{m}$,
we represent it  by a word $w=w_{1}w_{2}\ldots w_{n}$ such that $i\in B_{w_{i}}$, $1\leq i\leq n$.
We prove that the Mahonian statistics
{\footnotesize INV}, {\footnotesize MAJ}, {\footnotesize MAJ}$_{d}$, $r$-{\footnotesize MAJ}, {\footnotesize Z}, {\footnotesize DEN}, {\footnotesize MAK}, {\footnotesize MAD} are all equidistributed on set partitions via this representation,
and that the Euler-Mahonian statistics
$(\text{des}, \text{\footnotesize MAJ})$,
$(\text{mstc}, \text{\footnotesize INV})$,
$(\text{exc}, \text{\footnotesize DEN})$,
$(\text{des}, \text{\footnotesize MAK})$
are all equidistributed on set partitions via this representation.


   \vskip 2mm

\section{Introduction}
\subsection{Mahonian and Euler-Mahonian statistics}
We use the notation $M=\{1^{k_{1}},2^{k_{2}},\ldots,m^{k_{m}}\}$
for the multiset $M$ consisting of $k_{i}$ copies of $i$, for all $i\in[m]:=\{1,2,\ldots,m\}$.
Let $n=k_{1}+k_{2}+\cdots+k_{m}$, we write $|M|=n$ and $\langle M\rangle=m$.
Throughout this paper, we always assume that $M=\{1^{k_{1}},2^{k_{2}},\ldots,m^{k_{m}}\}$
with $k_{i}\geq1$ for all $i\in[m]$ and that $|M|=n$.
Let $\mathfrak{S}_{M}$ be the set of multipermutations  of multiset $M$.

Given $w=w_{1}w_{2}\ldots w_{n}\in\mathfrak{S}_{M}$,
a pair $(i,j)$ is called an \emph{inversion} of $w$ if $i<j$ and $w_{i}>w_{j}$.
Let $\text{\footnotesize INV}(w)$ be the number of inversions of $w$.
An index $i$, $1\leq i\leq n-1$, is called a \emph{descent} of $w$ if $w_{i}>w_{i+1}$.
Let $\text{Des}(w)$ be the set of all descents of $w$,
and let $\text{des}(w):=|\text{Des}(w)|$,
where $|\cdot|$ indicates cardinality.
Define the \emph{major index} of $w$, denoted $\text{\footnotesize MAJ}(w)$,
to be
\begin{align*}
\text{\footnotesize{MAJ}}(w)=\sum_{i\in {\text{Des}}(w)}i.
\end{align*}
MacMahon's equidistribution theorem asserts that for any multiset $M$,
\begin{align*}
\sum_{w\in\mathfrak{S}_{M}}q^{\text{\tiny INV}(w)}=\sum_{w\in\mathfrak{S}_{M}}q^{\text{\tiny MAJ}(w)}.
\end{align*}
In other words, the statistics {\footnotesize INV} and {\footnotesize MAJ} are equidistributed on $\mathfrak{S}_{M}$.
This famous result was obtained by MacMahon \cite{MacMahon-1916} in 1916.
It was not until 1968 that a famous bijective proof was found by Foata \cite{Foata-1968}.

Any statistic that is equidistributed with des is said to be \emph{Eulerian},
while any statistic equidistributed with \text{\footnotesize MAJ} is said to be \emph{Mahonian}.
A bivariate statistic that is equidistributed with (des, \text{\footnotesize MAJ}) is said to be \emph{Euler-Mahonian}.
Following \cite{Clarke-1997}, we will write Mahonian statistics with capital letters.
Tables \ref{Table-1} and \ref{Table-2} list the most common Mahonian and Euler-Mahonian statistics on words in the literature  respectively.
\begin{table}[t]
\centering
\begin{tabular}{lll}
Name& Reference & Year \\
\hline
{\footnotesize INV}&Rodriguez \cite{Rodriguez-1839}&1839\\
{\footnotesize MAJ}&MacMahon \cite{MacMahon-1916}&1916\\
$r$-{\footnotesize MAJ} &Rawlings \cite{Rawlings-1981} (for  permutations)& 1981\\
~ & Rawlings \cite{Rawlings-1981-2} (for words)& 1981  \\
{\footnotesize MAJ}$_{d}$&Kadell \cite{Kadell-1985}&1985\\
{\footnotesize Z}  &Zeilberger-Bressoud \cite{Zeilberger-1985}&1985  \\
{\footnotesize DEN} &Denert \cite {Denert-1990}, Foata-Zeilberger \cite{Foata-1990} (for  permutations)& 1990\\
~ & Han \cite{Han-1994} (for words)& 1994  \\
{\footnotesize MAK}  & Foata-Zeilberger \cite{Foata-1990} (for  permutations) & 1990 \\
~ &Clarke-Steingr\'{\i}msson-Zeng \cite{Clarke-1997} (for words) & 1997 \\
{\footnotesize MAD}  &Clarke-Steingr\'{\i}msson-Zeng \cite{Clarke-1997} &1997  \\
{\footnotesize STAT}  & Babson-Steingr\'{\i}msson \cite{Bjorner-2000} (for  permutations) &  2000 \\
~ & Kitaev-Vajnovszki \cite{Kitaev-2016} (for  words) &  2016 \\
\end{tabular}
\caption{~~Mahonian statistics on words.}\label{Table-1}
\end{table}

\begin{table}[t]
\centering
\begin{tabular}{lll}
Name& Reference & Year \\
\hline
$(\text{des}, \text{\footnotesize MAJ})$ &MacMahon \cite{MacMahon-1916}&1916\\
$(\text{exc}, \text{\footnotesize DEN})$ &Denert \cite {Denert-1990}, Foata-Zeilberger \cite{Foata-1990} (for  permutations)& 1990\\
~ &Han \cite{Han-1994} (for words)& 1994\\
$(\text{des}, \text{\footnotesize MAK})$ &Foata-Zeilberger \cite{Foata-1990} (for  permutations) & 1990 \\
&Clarke-Steingr\'{\i}msson-Zeng \cite{Clarke-1997} (for words) & 1997 \\
$(\text{mstc}, \text{\footnotesize INV})$ & Skandera \cite{Skandera-2001} (for  permutations) &  2001 \\
~ & Carnevale \cite{Carnevale-2017} (for  words) &  2017
\end{tabular}
\caption{~~Euler-Mahonian statistics on words.}\label{Table-2}
\end{table}

There are many research articles devoted to finding MacMahon type results for other combinatorial objects.
For example,
see
\cite{Chen-2010} for $01$-fillings of moon polyominoes,
\cite{Haglund-2006}  for standard Young tableaux,
\cite{Remmel-2015} for ordered set partitions,
\cite{Liu-2021}  for $k$-Stirling permutations.
In this paper we consider set partitions.

\subsection{Set partitions}\label{subsection-set partition}
A \emph{partition of the set} $[n]=\{1,2,\ldots,n\}$ is a collection of disjoint nonempty subsets (or blocks) of $[n]$, whose union is $[n]$.
For example, $\{\{1,3,5,7\},\{2,6\},\{4\},\{8,9\}\}$ is a partition of $[9]$.
We denote by $\Pi_{n}$ the set of all partitions of $[n]$,
and by $\Pi_{n,m}$ the set of all partitions of $[n]$ with exactly $m$ blocks.
There are several well-known representations for set partitions,
and each of them has its aim value and its result.
In the following, we provide three most common representations.

Given a partition of $[n]$,
the graph on the vertex set $[n]$ whose edge set consists of the arcs connecting the elements of each block in numerical order  is called the \emph{standard representation}.
For example, the standard representation of  $\{\{1,3,5,7\},\{2,6\},\{4\},\{8,9\}\}$ has the arc set $\{(1,3),(3,5),(5,7),(2,6),(8,9)\}$;
see Fig. \ref{Fig_exm_standard_representation}.
Using the standard representation,
Chen, Gessel, Yan and Yang \cite{Chen-2008} introduced a major index statistic and then obtained a MacMahon type result for set partitions.

\begin{figure}[t]
\centering
\includegraphics[width=9cm]{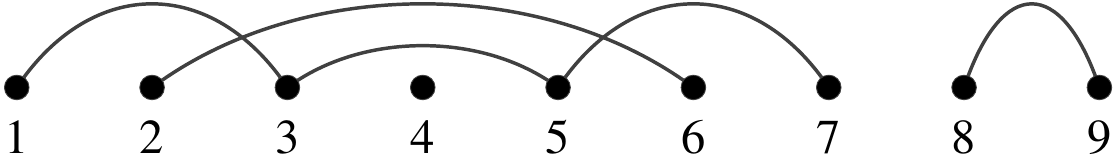}
\caption{\quad The standard representation of $\{\{1,3,5,7\},\{2,6\},\{4\},\{8,9\}\}$.
}\label{Fig_exm_standard_representation}
\end{figure}

Given a partition of $[n]$,
write it as $B_{1}/B_{2}/\cdots/B_{m}$,
where $B_{1},B_{2},\ldots,B_{m}$ are the blocks and satisfy the
following order property
$$\min B_{1}<\min B_{2}<\cdots<\min B_{m}.$$
This representation is called  the \emph{block representation}.
For example, the block representation of $\{\{4\},\{2,6\},\{8,9\},\{1,3,5,7\}\}$ is $\{1,3,5,7\}/\{2,6\}/\{4\}/\{8,9\}$.
From the historical point of view, the set partitions are defined by block representation.
Using the block representation,
Sagan \cite{Sagan-1991} introduced a major index statistic and then obtained a  MacMahon type result for set partitions.

Given a partition $w$ of $[n]$ with blocks $B_{1},B_{2},\cdots,B_{m}$,
where $$\min B_{1}<\min B_{2}<\cdots<\min B_{m},$$
we write $w=w_{1}w_{2}\ldots w_{n}$,
where $w_{i}$ is the block number in which $i$ appears,
that is,  $i\in B_{w_{i}}$.
This representation is called the \emph{canonical representation}.
For example, the canonical representation of $\{\{1,3,5,7\},\{2,6\},\{4\},\{8,9\}\}$ is $121312144$.
The canonical representation is one of the most popular representations in the theory of set partitions.
Many statistics,
especially the pattern-based statistics,
on set partitions are defined via canonical representation,
see, e.g., the book \cite{Mansour-2013} for more information.

In this paper, we consider the following rarely used representation for set partitions:
given a partition $w$ of $[n]$ with blocks $B_{1},B_{2},\ldots,B_{m}$,
where $$\max B_{1}<\max B_{2}<\cdots<\max B_{m},$$
we write $w=w_{1}w_{2}\ldots w_{n}$,
where $w_{i}$ is the block number in which $i$ appears,
that is,  $i\in B_{w_{i}}$.
Note that the only difference between this representation and canonical representation is
that here we use the ordering according to the \emph{maximal} element of the blocks.
This representation was also considered by Johnson \cite{Johnson-1996-1,Johnson-1996-2} and Deodhar-Srinivasan \cite{Deodhar-2003}.
In this paper we prove that the Mahonian statistics on words in Table \ref{Table-1}, except for {\footnotesize STAT}, are all equidistributed on set partitions via this representation,
and that the Euler-Mahonian statistics on words in Table \ref{Table-2} are all equidistributed on set partitions via this representation.
For this reason,
we call this representation the \emph{Mahonian representation} for set partitions.
We will state formally the main results of this paper in the next section,
and will prove them in the remaining sections.

\section{Main results}
\subsection{Definitions and the main results}
Given $w=w_{1}w_{2}\ldots w_{n}\in\mathfrak{S}_{M}$,
define the \emph{tail permutation of $w$} to be the subword $w_{t_{1}}w_{t_{2}}\ldots w_{t_{m}}$ of $w$, where $t_{1}<t_{2}<\ldots<t_{m}$, so that $w_{t_{i}}$ is the last (rightmost) occurrence of that letter for any $i\in[m]$.
For example, the tail permutation of $331322112441$ is $3241$.
Clearly the tail permutation of $w\in\mathfrak{S}_{M}$ is a permutation of $[m]$ (recall that we always assume that $\langle M\rangle=m$).

Let $\mathfrak{S}_{M}^{\tau}$ be the set of words in $\mathfrak{S}_{M}$ with tail permutation $\tau$.
In particular, let $\mathcal{P}_{M}$ be the set of words in $\mathfrak{S}_{M}$ with increasing  tail permutation, that is, $\mathcal{P}_{M}=\mathfrak{S}_{M}^{\tau}$, where $\tau=12\cdots m$.
Thus, $\mathcal{P}_{M}$ is the set of words in $\mathfrak{S}_{M}$
in which the last occurrences of $1,2,\ldots,m$ occur in that order.
For example, let $M=\{1^{2},2^{2}\}$, then
$\mathfrak{S}_{M}=\{1122,1212,1221,2112,2121,2211\}$,
$\mathfrak{S}_{M}^{21}=\{1221,2121,2211\}$, and
$\mathcal{P}_{M}=\mathfrak{S}_{M}^{12}=\{1122,1212,2112\}$.

Let $w$ be a partition of $[n]$ with blocks $B_{1},B_{2},\ldots,B_{m}$,
where
$$\max B_{1}<\max B_{2}<\cdots<\max B_{m},$$
if $|B_{i}|=k_{i}$, $1\leq i\leq m$, we say that $w$ is of \emph{type} $(k_{1},k_{2},\ldots,k_{m})$.
It is obvious that $\mathcal{P}_{M}$ is the set of the partitions of $[n]$ of type $(k_{1},k_{2},\ldots,k_{m})$ via Mahonian representation.
Throughout this paper we always use the Mahonian representation to represent a set partition,
that is, we always think of a set partition as an element in $\mathcal{P}_{M}$ for some $M$.
Then
\begin{align*}
\Pi_{n}=\bigcup_{|M|=n}\mathcal{P}_{M} \quad \text{and} \quad \Pi_{n,m}=\bigcup_{|M|=n,\langle M\rangle=m}\mathcal{P}_{M}.
\end{align*}
The main results of this paper can be summarized as the following three theorems.
\begin{theorem}\label{Thm-MacMahon-set-partition}
Let $M=\{1^{k_{1}},2^{k_{2}},\ldots,m^{k_{m}}\}$ with $k_{i}\geq1$ for all $i\in[m]$,
then
\begin{align}\label{Main-MacMahon-eq-1}
\begin{split}
&\sum_{w\in\mathcal{P}_{M}}q^{\emph{\tiny INV}(w)}~~=
\sum_{w\in\mathcal{P}_{M}}q^{\emph{\tiny MAJ}(w)}=\sum_{w\in\mathcal{P}_{M}}q^{\emph{\tiny MAJ}_{d}(w)}=
\sum_{w\in\mathcal{P}_{M}}q^{\emph{\tiny Z}(w)}\\=
&\sum_{w\in\mathcal{P}_{M}}q^{r\emph{\tiny -MAJ}(w)}
=\sum_{w\in\mathcal{P}_{M}}q^{\emph{\tiny DEN}(w)}=
\sum_{w\in\mathcal{P}_{M}}q^{\emph{\tiny MAK}(w)}=
\sum_{w\in\mathcal{P}_{M}}q^{\emph{\tiny MAD}(w)}.
\end{split}
\end{align}
\end{theorem}
Note that (\ref{Main-MacMahon-eq-1}) shows that the Mahonian statistics on words in Table \ref{Table-1}, except for {\footnotesize STAT}, are all equidistributed on set partitions of given type via Mahonian representation.
The following theorem shows that the Euler-Mahonian statistics on words in Table \ref{Table-2} are all equidistributed on set partitions of given type via Mahonian representation.

\begin{theorem}\label{Thm-Euler-MacMahon-set-partition}
Let $M=\{1^{k_{1}},2^{k_{2}},\ldots,m^{k_{m}}\}$ with $k_{i}\geq1$ for all $i\in[m]$,
then
\begin{align*}
\begin{split}
\sum_{w\in\mathcal{P}_{M}}t^{\emph{des}(w)}q^{\emph{\tiny MAJ}(w)}=
\sum_{w\in\mathcal{P}_{M}}t^{\emph{mstc}(w)}q^{\emph{\tiny INV}(w)}=
\sum_{w\in\mathcal{P}_{M}}t^{\emph{exc}(w)}q^{\emph{\tiny DEN}(w)}=
\sum_{w\in\mathcal{P}_{M}}t^{\emph{des}(w)}q^{\emph{\tiny MAK}(w)}.
\end{split}
\end{align*}
\end{theorem}

A permutation $\tau=\tau_{1}\tau_{2}\ldots\tau_{m}$ of $[m]$ is said to be \emph{consecutive} if $\{\tau_{1},\tau_{2},\ldots,\tau_{i}\}$ forms a set of consecutive numbers for all $i\in[m]$.
For example, $\tau=54362718$ is consecutive,
whereas $\tau=54236718$ is not.
In particular, the increasing permutation $\tau=12\ldots m$ is consecutive.
Throughout this paper we always assume that $\tau$ is a consecutive permutation.
For the  statistics {\footnotesize INV}, {\footnotesize MAJ}, {\footnotesize MAJ}$_{d}$, {\footnotesize Z},
we prove the following  more general result.
\begin{theorem}\label{Thm-inv-maj-set-partition-general}
Let $M=\{1^{k_{1}},2^{k_{2}},\ldots,m^{k_{m}}\}$ with $k_{i}\geq1$ for all $i\in[m]$,
and let $\tau$ be a consecutive permutation of $[m]$,
then
\begin{align}\label{Main-general-eq-1}
\sum_{w\in\mathfrak{S}_{M}^{\tau}}q^{\emph{\tiny INV}(w)}=
\sum_{w\in\mathfrak{S}_{M}^{\tau}}q^{\emph{\tiny MAJ}(w)}=
\sum_{w\in\mathfrak{S}_{M}^{\tau}}q^{\emph{\tiny MAJ}_{d}(w)}=
\sum_{w\in\mathfrak{S}_{M}^{\tau}}q^{\emph{\tiny Z}(w)}.
\end{align}
\end{theorem}
Setting $\tau=12\ldots m$ in (\ref{Main-general-eq-1}) gives the top row of (\ref{Main-MacMahon-eq-1}).
\begin{remark}
\emph{
Small examples show that  none of $r$-{\footnotesize MAJ}, {\footnotesize DEN}, {\footnotesize MAK} and {\footnotesize MAD}  is  equidistributed with {\footnotesize INV} (or {\footnotesize MAJ}, {\footnotesize MAJ}$_{d}$, {\footnotesize Z}) on $\mathfrak{S}_{M}^{\tau}$ for consecutive $\tau$ in general.
}

\end{remark}
\begin{remark}
\emph{Small examples show that any two of
$(\text{des}, \text{\footnotesize MAJ})$, $(\text{mstc}, \text{\footnotesize INV})$,
$(\text{exc}, \text{\footnotesize DEN})$ and
$(\text{des}, \text{\footnotesize MAK})$
are \emph{not} equidistributed on $\mathfrak{S}_{M}^{\tau}$ for consecutive $\tau$ in general.
}\end{remark}

To conclude this subsection,
we give an example.
Table \ref{Table-3} gives the distributions of the  statistics {\footnotesize INV}, {\footnotesize MAJ}, {\footnotesize MAJ}$_{2}$, {\footnotesize Z}, $2$-{\footnotesize MAJ}, {\footnotesize DEN}, {\footnotesize MAK}, {\footnotesize MAD}, {\footnotesize STAT} on $\mathcal{P}_{\{1,1,2,2,3,3\}}$.
We see that
\begin{align*}
\begin{split}
&\sum_{w\in\mathcal{P}_{\{1,1,2,2,3,3\}}}q^{\text{\tiny STAT}(w)}=
1+q^{2}+2q^{3}+3q^{4}+2q^{5}+q^{6}+2q^{7}+2q^{8}+q^{9},\\
&\sum_{w\in\mathcal{P}_{\{1,1,2,2,3,3\}}}~q^{S(w)}~~=1+2q+3q^{2}+3q^{3}+3q^{4}+2q^{5}+q^{6},
\end{split}
\end{align*}
where $S$ is any of {\footnotesize INV}, {\footnotesize MAJ}, {\footnotesize MAJ}$_{2}$, {\footnotesize Z}, $2$-{\footnotesize MAJ}, {\footnotesize DEN}, {\footnotesize MAK}, {\footnotesize MAD}.
\begin{table}[t]
\centering
\begin{tabular}{l|ccccccccc}
\hline
$\mathcal{P}_{\{1,1,2,2,3,3\}}$
&{\footnotesize INV}\!
&{\footnotesize MAJ}
&{\footnotesize MAJ}$_{2}$
&{\footnotesize Z}
&{\footnotesize 2-MAJ}
&{\footnotesize DEN}
&{\footnotesize MAK}
&{\footnotesize MAD}
&{\footnotesize STAT}
\\
\hline
112233 &0 & 0 &0 & 0 &0 & 0  &0 &0 &0\\
112323 &1 & 4 &1 & 2 &1 & 4  &4 &1 &3\\
113223 &2 & 3 &4 & 1 &2 & 3  &3 &2 &4\\
121233 &1 & 2 &1 & 2 &1 & 2  &2 &1 &5\\
121323 &2 & 6 &2 & 4 &2 & 6  &6 &2 &8\\
123123 &3 & 3 &6 & 6 &5 & 5  &3 &4 &4\\
131223 &3 & 2 &3 & 3 &4 & 2  &2 &3 &5\\
132123 &4 & 5 &4 & 5 &3 & 3  &5 &3 &9\\
211233 &2 & 1 &2 & 1 &2 & 1  &1 &2 &4\\
211323 &3 & 5 &3 & 3 &3 & 5  &5 &3 &7\\
213123 &4 & 4 &5 & 5 &6 & 4  &4 &6 &8\\
231123 &5 & 2 &4 & 4 &5 & 3  &2 &5 &3\\
311223 &4 & 1 &2 & 2 &3 & 1  &1 &4 &2\\
312123 &5 & 4 &3 & 4 &4 & 2  &3 &5 &6\\
321123 &6 & 3 &5 & 3 &4 & 4  &4 &4 &7\\

\hline
\end{tabular}
\caption{The distributions of the  statistics {\footnotesize INV}, {\footnotesize MAJ}, {\footnotesize MAJ}$_{2}$, {\footnotesize Z}, $2$-{\footnotesize MAJ}, {\footnotesize DEN}, {\footnotesize MAK}, {\footnotesize MAD}, {\footnotesize STAT} on $\mathcal{P}_{\{1,1,2,2,3,3\}}$.}\label{Table-3}
\end{table}
\subsection{$q$-Stirling numbers of the second kind}
A $q$-analog of a mathematical object is an object depending on the variable $q$ that reduces to the original object when we set $q=1$.
Let
\begin{align*}
[n]_{q}=1+q+q^{2}+\cdots+q^{n-1},~\quad
[n]_{q}!=[1]_{q}[2]_{q}\ldots[n]_{q},~\quad
{n \choose i}_{q}=\frac{[n]_{q}!}{[i]_{q}![n-i]_{q}!}.
\end{align*}
Then $[n]_{q}$, $[n]_{q}!$ and ${n \choose i}_{q}$ are the $q$-analogs of $n$, $n!$ and ${n \choose i}$ respectively.

It is well-known that
the total number of partitions of $[n]$ is the \emph{Bell number}  $B(n)$,
and that
the number of partitions of $[n]$ with exactly $m$ blocks is the \emph{Stirling number of the second kind} $S(n,m)$.
That is,
$$|\Pi_{n}|=B(n),~\quad|\Pi_{n,m}|=S(n,m).$$
There are two recursions for $S(n,m)$:
\begin{align}\label{Stirling numbers recursion 1}
S(n,m)=S(n-1,m-1)+mS(n-1,m), \quad S(0,m)=\delta_{0,m},
\end{align}
and
\begin{align}\label{Stirling numbers recursion 2}
S(n+1,m)=\sum_{i=0}^{n}{n \choose i}S(n-i,m-1), \quad S(0,m)=\delta_{0,m},
\end{align}
where $\delta_{n,m}$ is the Kronecker delta, defined by $\delta_{n,m}=1$ if $n=m$, and $\delta_{n,m}=0$ otherwise.
Considering the $q$-analogs of the above two recursions leads to two different kinds of $q$-Stirling numbers of the second kind.

\emph{Carlitz's $q$-Stirling numbers of the second kind}, denoted $S_{q}(n,m)$,  are defined by the following recursion
$$S_{q}(n,m)=S_{q}(n-1,m-1)+[m]_{q}S_{q}(n-1,m), \quad\quad S_{q}(0,m)=\delta_{0,m},$$
which is a $q$-analog of (\ref{Stirling numbers recursion 1}).
These polynomials were first studied by Carlitz \cite{Carlitz-1933,Carlitz-1948}
and then Gould \cite{Gould-1961}.
Milne \cite{Milne-1982} introduced an inversion statistic, here we denote  $\underline{\text{inv}}$, for set partitions via canonical representation,
and he proved
$$S_{q}(n,m)=\sum_{w\in\Pi_{n,m}}q^{\underline{\text{inv}}(w)}.$$
Sagan \cite{Sagan-1991} introduced a major index statistic, here we denote  $\underline{\text{maj}}$,
for set partitions via block representation,  and he proved
$$S_{q}(n,m)=\sum_{w\in\Pi_{n,m}}q^{\underline{\text{maj}}(w)}.$$

\emph{Johnson's $q$-Stirling numbers of the second kind}, denoted  {\scriptsize$\left\{
 \begin{matrix}
   n\\
   m
  \end{matrix}
  \right\}_{q}$},
  are defined by the following recursion
$$\left\{\begin{matrix}
   n+1\\
   m
  \end{matrix}
  \right\}_{q}=
  \sum_{i=0}^{n}
  \left(
  \begin{matrix}
   n\\
   i
  \end{matrix}
  \right)_{q}
  \left\{
 \begin{matrix}
   n-i\\
   m-1
  \end{matrix}
  \right\}_{q},\quad\quad \left\{\begin{matrix}
   0\\
   m
  \end{matrix}
  \right\}_{q}=\delta_{0,m},$$
which is a $q$-analog of (\ref{Stirling numbers recursion 2}).
These polynomials were first studied by Johnson \cite{Johnson-1996-2}.
As pointed out by Johnson \cite{Johnson-1996-2}, {\scriptsize$\left\{
 \begin{matrix}
   n\\
   m
  \end{matrix}
  \right\}_{q}$}
is different from $S_{q}(n,m)$.
Johnson \cite{Johnson-1996-2} proved
$$\quad\left\{\begin{matrix}
   n\\
   m
  \end{matrix}
  \right\}_{q}=\sum_{w\in\Pi_{n,m}}q^{\text{\tiny INV}(w)},$$
where $w$ uses the Mahonian representation.
Combining Johnson's result with Theorem \ref{Thm-MacMahon-set-partition},
we have the following corollary.
\begin{corollary}\label{Cor-q-Stirling number}
Using the Mahonian representation, we have
$$\quad\left\{\begin{matrix}
   n\\
   m
  \end{matrix}
  \right\}_{q}=\sum_{w\in\Pi_{n,m}}q^{S(w)},$$
where $S$ is any of \emph{{\footnotesize INV}, {\footnotesize MAJ}, {\footnotesize MAJ}$_{d}$, $r$-{\footnotesize MAJ}, {\footnotesize Z}, {\footnotesize DEN}, {\footnotesize MAK}, {\footnotesize MAD}}.
\end{corollary}

\subsection{Structure of the paper}
This paper is organized as follows.
In Sections \ref{section-inv-maj}, \ref{section-inv-majd} and \ref{section-maj-Z},
we establish the equidistributions of
{\footnotesize INV} and {\footnotesize MAJ},
{\footnotesize INV} and {\footnotesize MAJ$_{d}$},
{\footnotesize MAJ} and {\footnotesize Z}
on $\mathfrak{S}_{M}^{\tau}$ for consecutive $\tau$, respectively.
In Section  \ref{section-bi-1}, we establish the equidistribution of
the bi-statistics $(\text{mstc}, \text{{\footnotesize INV}})$ and $(\text{des}, \text{{\footnotesize MAJ}})$ on $\mathcal{P}_{M}$.
In Section \ref{section-inv-rmaj}, we  establish  the equidistribution of {\footnotesize INV} and $r$-{\footnotesize MAJ} on $\mathcal{P}_{M}$.
In Section  \ref{section-bi-2}, we establish the equidistribution of the bi-statistics $(\text{des}, \text{\footnotesize MAJ})$ and $(\text{exc}, \text{\footnotesize DEN})$ on $\mathcal{P}_{M}$.
In Section  \ref{section-tri}, we  establish  the equidistribution of the triple statistics
(des, {\footnotesize MAK}, {\footnotesize MAD}) and (exc, {\footnotesize DEN}, {\footnotesize INV}) on $\mathcal{P}_{M}$.
Combining these results
we obtain Theorems \ref{Thm-MacMahon-set-partition}, \ref{Thm-Euler-MacMahon-set-partition} and \ref{Thm-inv-maj-set-partition-general}.

\section{Equidistribution of {\large INV} and {\large MAJ} on $\mathfrak{S}_{M}^{\tau}$}\label{section-inv-maj}
The equidistribution of {\footnotesize INV} and {\footnotesize MAJ} was proved bijectively
for the first time by Foata \cite{Foata-1968}.
We denote this famous bijection by $F$.
The goal of this section is to prove that the classical statistics {\footnotesize INV} and {\footnotesize MAJ} are equidistributed on $\mathfrak{S}_{M}^{\tau}$ for consecutive $\tau$.
We will show this by proving that Foata's bijection $F$ preserves the consecutive tail permutation.
We first give a brief description of $F$.

Given a word $w=w_{1}w_{2}\ldots w_{n}$ and a letter $x$.
We first define an operator $J_{x}$ on $w$.
If $w_{n}\leq x$,
write $w=u_{1}b_{1}u_{2}b_{2}\ldots u_{s}b_{s}$,
where each $b_{i}$ is a letter less than or equal to $x$,
and each $u_{i}$ is a word (possibly empty),
all of whose letters are greater than $x$.
Similarly, if $w_{n}> x$,
write $w=u_{1}b_{1}u_{2}b_{2}\ldots u_{s}b_{s}$,
where each $b_{i}$ is a letter greater than $x$,
and each $u_{i}$ is a word (possibly empty),
all of whose letters  less than or equal to $x$.
In each case we define the operator $J_{x}$ on $w$ to be
$$J_{x}(w)=b_{1}u_{1}b_{2}u_{2}\ldots b_{s}u_{s}.$$
We call the letters $b_{1},b_{2},\ldots,b_{s}$  the \emph{jumping letters},
the remaining letters  the \emph{fixed letters}.
The following  property for the operator $J_{x}$ is crucial
\begin{equation*}
\text{\footnotesize INV}(J_{x}(w)x)-\text{\footnotesize INV}(w)=\left\{
\begin{aligned}
&n, \quad\quad\text{if~}w_{n}>x, \\
&0, \quad\quad\text{if~}w_{n}\leq x.
\end{aligned}
\right.
\end{equation*}
Given $w=w_{1}w_{2}\ldots w_{n}$.
Define $\gamma_{1}=w_{1}$,
and $\gamma_{i+1}=J_{w_{i+1}}(\gamma_{i})w_{i+1}$ for $1\leq i\leq n-1$.
Finally set $F(w)=\gamma_{n}$.

\begin{theorem}[Foata \cite{Foata-1968}]\label{Foata}
$F:\mathfrak{S}_{M}\rightarrow\mathfrak{S}_{M}$ is a bijection
satisfying $$\emph{\footnotesize{MAJ}}(w)=\emph{\footnotesize{INV}}(F(w))$$ for all $w\in\mathfrak{S}_{M}$.
\end{theorem}
\begin{example}
\emph{Let $w=211323$,
it is not hard to see that $\text{\footnotesize{MAJ}}(w)=5$.
We give the  procedure for creating $F(w)$:
\begin{align*}
 \gamma_{1}&=2,\\
 \gamma_{2}&=J_{1}(2)1=21, \\
 \gamma_{3}&=J_{1}(21)1=121, \\
 \gamma_{4}&=J_{3}(121)3=1213,\\
 \gamma_{5}&=J_{2}(1213)2=31212, \\
 \gamma_{6}&=J_{3}(31212)3=312123=F(w).
\end{align*}
Note that $\text{\footnotesize{INV}}(F(w))=5$.}
\end{example}

The following result shows that Foata's bijection preserves the consecutive tail permutation.

\begin{proposition}\label{Prop-Foata-closed-on-M}
Let $M=\{1^{k_{1}},2^{k_{2}},\ldots,m^{k_{m}}\}$ with $k_{i}\geq1$ for all $i\in[m]$,
and let $\tau$ be a consecutive permutation of $[m]$.
Then the set $\mathfrak{S}_{M}^{\tau}$ is invariant under Foata's bijection, that is
$$F(\mathfrak{S}_{M}^{\tau})=\mathfrak{S}_{M}^{\tau}.$$
\end{proposition}
\begin{proof}
To prove $F(\mathfrak{S}_{M}^{\tau})=\mathfrak{S}_{M}^{\tau}$,
it suffices to show that $F(w)\in\mathfrak{S}_{M}^{\tau}$ for all $w\in\mathfrak{S}_{M}^{\tau}$,
because $F$ is a  bijection.
Suppose that $\tau=\tau_{1}\tau_{2}\ldots\tau_{m}$.

Given a word  with some letters overlined,
we call the subword consisting of the overlined letters the \emph{overlined subword}.
For example,
the overlined subword of $331\overline{3}2211\overline{2}4\overline{4}\overline{1}$ is $\overline{3241}$.

Given $w=w_{1}w_{2}\ldots w_{n}\in\mathfrak{S}_{M}^{\tau}$,
we overline the last occurrences of the letters $1,2,\ldots,m$.
Thus, the overlined subword of $w$ is $\overline{\tau_{1}\tau_{2}\ldots\tau_{m}}$ as $w\in\mathfrak{S}_{M}^{\tau}$.
Note that $\gamma_{i}$ is a word with some letters overlined for $1\leq i\leq n$,
because $\gamma_{i}$ is a rearrangement of $w_{1}w_{2}\ldots w_{i}$.
We claim that for any $i$, $1\leq i\leq n$, we have
\begin{itemize}
\item[(i)] the overlined subword of $\gamma_{i}$ is the same as the overlined subword of $w_{1}w_{2}\ldots w_{i}$.
\item[(ii)] each overlined letter of $\gamma_{i}$ is the last occurrence of that letter in $\gamma_{i}$.

\end{itemize}

We use induction on $i$ to prove our claim.
The initial case of $i=1$ is obvious.
Assume that our claim  is true for $i$ and prove it for $i+1$, where $1\leq i\leq n-1$.
Suppose that the overlined subword of $w_{1}w_{2}\ldots w_{i}$ is $\overline{\tau_{1}\tau_{2}\ldots\tau_{s}}$.
For $1\leq k\leq s$,
because $\overline{\tau}_{k}$, which is the last occurrence of $\tau_{k}$ in $w$, appears in $w_{1}w_{2}\ldots w_{i}$,
we have that $w_{i+1}\notin\{\tau_{1},\tau_{2},\ldots,\tau_{s}\}$.

Below we assume that $\gamma_{i}=c_{1}c_{2}\ldots c_{i}$.
(This will simplify the description of the proof of Proposition \ref{Prop-Foata-d-closed-on-M} in the next section although it is not necessary.)

(A) We first prove (i) for $i+1$.
By the induction hypothesis, the overlined subword of $\gamma_{i}$ is $\overline{\tau_{1}\tau_{2}\ldots\tau_{s}}$.
Since $\{\tau_{1},\tau_{2},\ldots,\tau_{s}\}$ is a set of consecutive numbers and $w_{i+1}\notin\{\tau_{1},\tau_{2},\ldots,\tau_{s}\}$,
we see that either all of $\tau_{1},\tau_{2},\ldots,\tau_{s}$ are greater than $w_{i+1}$,
or all of them are smaller than $w_{i+1}$.
It follows that  either all of the overlined letters of $\gamma_{i}$ are  greater than $w_{i+1}$, or all of them are smaller than $w_{i+1}$.
Then either all of the overlined letters of $c_{1}c_{2}\ldots c_{i}$  are jumping letters,
or all of them are fixed letters.
Combining this with the definition of the operator $J_{w_{i+1}}$,
we can see that the operator $J_{w_{i+1}}$ dose not change the overlined subword of $c_{1}c_{2}\ldots c_{i}$.
Then by the definition of $\gamma_{i+1}$,
we have that the overlined subword of $\gamma_{i+1}$ is the same as overlined subword of $w_{1}w_{2}\ldots w_{i}w_{i+1}$,
this  completes the proof of (i).

(B) We now prove (ii) for $i+1$.
For $1\leq k\leq s$,
by the induction hypothesis
we see that $\overline{\tau}_{k}$ is the last occurrence of the letter $\tau_{k}$ in $\gamma_{i}$.
It is not hard to see that
either all the occurrences of $\tau_{k}$ in $c_{1}c_{2}\ldots c_{i}$ are jumping letters,
or all of them are fixed letters.
Combining this with the fact that $w_{i+1}\notin\{\tau_{1},\tau_{2},\ldots,\tau_{s}\}$,
we obtain that $\overline{\tau}_{k}$ is the last occurrence of the letter $\tau_{k}$ in $\gamma_{i+1}$, $1\leq k\leq s$.
If $w_{i+1}$ is not overlined, we complete the proof of (ii) for $i+1$.
If $w_{i+1}$ is overlined, it must be the last occurrence of that letter in $\gamma_{i+1}$,
we also complete the proof of (ii) for $i+1$.

By (i) of our claim,
the overlined subword of $\gamma_{n}=F(w)$ is $\overline{\tau_{1}\tau_{2}\ldots\tau_{m}}$.
By (ii) of our claim,
the tail permutation of $F(w)$ is
$\tau_{1}\tau_{2}\ldots\tau_{m}$.
So $F(w)\in\mathfrak{S}_{M}^{\tau}$,
completing the proof.
\end{proof}

\begin{remark}
\emph{
It is not hard to prove that if $\tau$ is a consecutive permutation of $[m]$, then
$F(\tau)=\tau$.
(This can also be deduced by a theorem of Bj\"{o}rner and Wachs \cite[Theorem 4.2]{Bjorner-1991}.
Also see \cite{Chen-2011}.)
Proposition \ref{Prop-Foata-closed-on-M} can be viewed a generalization of this result to words,
because Proposition \ref{Prop-Foata-closed-on-M} gives the above result when we set $M=\{1,2,\ldots,m\}$.
}
\end{remark}
By Theorem \ref{Foata} and Proposition \ref{Prop-Foata-closed-on-M} we obtain the equidistribution of {\footnotesize INV} and {\footnotesize MAJ} on $\mathfrak{S}_{M}^{\tau}$
for consecutive $\tau$,
and we  achieve the goal of this section.

\section{Equidistribution of {\large INV} and {\large MAJ}$_{d}$  on $\mathfrak{S}_{M}^{\tau}$}\label{section-inv-majd}
Given $w=w_{1}w_{2}\ldots w_{n}\in\mathfrak{S}_{M}$,
let $d$ be a positive integer,
define
\begin{align*}
\text{\footnotesize{MAJ}}_{d}(w)=\text{inv}_{d}(w)+\sum_{
w_{i}>w_{i+d}}i,
\end{align*}
where
$$\text{inv}_{d}(w)=|\{(i,j):i<j<i+d,~w_{i}>w_{j}\}|.$$
It is not hard to see that
$\text{\footnotesize{MAJ}}_{1}=\text{\footnotesize{MAJ}}$ and  $\text{\footnotesize{MAJ}}_{n}=\text{\footnotesize{INV}}$,
thus the family of $\text{\footnotesize{MAJ}}_{d}$ interpolates between {\footnotesize{MAJ}} and {\footnotesize{INV}}.
The statistic {\footnotesize MAJ}$_{d}$ was introduced by Kadell \cite{Kadell-1985},
who gave a bijective proof that this statistic is Mahonian.
Kadell's bijection  takes {\footnotesize INV} to {\footnotesize MAJ}$_{d}$,
with the extreme case taking {\footnotesize INV} to
{\footnotesize MAJ} corresponding precisely to the inverse of Foata's bijection.
Assaf \cite{Assaf-2008} exhibited a different family of bijections,
taking {\footnotesize MAJ}$_{d-1}$  to {\footnotesize MAJ}$_{d}$.
In his Ph.D. thesis \cite{Liang-1989},
Liang gave a bijection $F_{d}$ that takes {\footnotesize MAJ}$_{d}$ to {\footnotesize INV}
(the author independently found this bijection, and later discovered the Ph.D. thesis of Liang \cite{Liang-1989}).
The bijection $F_{d}$ is a natural extension of Foata's bijection and is the inverse of Kadell's bijection.

We now describe the bijection $F_{d}$.
Given $w=w_{1}w_{2}\ldots w_{n}\in\mathfrak{S}_{M}$.
We will define words $\gamma_{1},\gamma_{2},\ldots,\gamma_{n}$,
where $\gamma_{i}$ is a rearrangement of $w_{1}w_{2}\ldots w_{i}$.
First define $\gamma_{i}=w_{1}w_{2}\ldots w_{i}$ for $1\leq i\leq d$.
Assume that $\gamma_{i}=c_{1}c_{2}\ldots c_{i}$ has been defined for some $d\leq i\leq n-1$.
Then define
\begin{align}\label{Eq-Foata-bijection-eq-extension}
\gamma_{i+1}=J_{w_{i+1}}(c_{1}c_{2}\ldots c_{i-d+1})c_{i-d+2}c_{i-d+3}\ldots c_{i}w_{i+1}.
\end{align}
Finally, set $F_{d}(w)=\gamma_{n}$.
When $d=1$, we have $\gamma_{i+1}=J_{w_{i+1}}(\gamma_{i})w_{i+1}.$
Thus $F_{d}$ reduces to Foata's bijection $F$ when $d=1$.
We have the following theorem.
\begin{theorem}[Liang \cite{Liang-1989}]\label{Thm-F_d-bijection}
$F_{d}:\mathfrak{S}_{M}\rightarrow\mathfrak{S}_{M}$ is a bijection
satisfying
$$\emph{\footnotesize{MAJ}}_{d}(w)=\emph{\footnotesize{INV}}(F_{d}(w))$$ for all $w\in\mathfrak{S}_{M}$.
\end{theorem}
\begin{example}
\emph{Let $w=213123$ and $d=2$,
it is not hard to see that $\text{\footnotesize{MAJ}}_{2}(w)=5$.
We give the  procedure for creating $F_{2}(w)$:
\begin{align*}
 \gamma_{1}&=2,\\
 \gamma_{2}&=21, \\
 \gamma_{3}&=J_{3}(2)13=213, \\
 \gamma_{4}&=J_{1}(21)31=1231,\\
 \gamma_{5}&=J_{2}(123)12=31212, \\
 \gamma_{6}&=J_{3}(3121)23=312123=F_{2}(w).
\end{align*}
Note that $\text{\footnotesize{INV}}(F_{2}(w))=5$.}
\end{example}

The following proposition shows  that $F_{d}$ preserves the consecutive tail permutation,
which is a generalization of Proposition \ref{Prop-Foata-closed-on-M}.
\begin{proposition}\label{Prop-Foata-d-closed-on-M}
Let $M=\{1^{k_{1}},2^{k_{2}},\ldots,m^{k_{m}}\}$ with $k_{i}\geq1$ for all $i\in[m]$,
and let $\tau$ be a consecutive permutation of $[m]$,
then
$\mathfrak{S}_{M}^{\tau}$ is invariant under $F_{d}$,
that is,
$$F_{d}(\mathfrak{S}_{M}^{\tau})=\mathfrak{S}_{M}^{\tau}.$$
\end{proposition}
Replacing each $c_{1}c_{2}\ldots c_{i}$  by $c_{1}c_{2}\ldots c_{i+d-1}$
in the paragraphs (A) and (B) of the proof of  Proposition \ref{Prop-Foata-closed-on-M}
gives the proof of Proposition \ref{Prop-Foata-d-closed-on-M}.

By Theorem \ref{Thm-F_d-bijection} and Proposition \ref{Prop-Foata-d-closed-on-M},
we obtain the equidistribution of {\footnotesize INV} and $\text{{\footnotesize MAJ}}_{d}$ on $\mathfrak{S}_{M}^{\tau}$
for consecutive $\tau$,
and we  achieve the goal of this section.

\section{Equidistribution of {\large MAJ} and {\large Z} on $\mathfrak{S}_{M}^{\tau}$}\label{section-maj-Z}
Given a word $w=w_{1}w_{2}\ldots w_{n}$,
the \emph{$z$-index} of $w$, denoted by {\footnotesize Z}$(w)$, is defined by
$$\text{\footnotesize Z}(w)=\sum_{i<j}\text{\footnotesize MAJ}(w_{ij}),$$
where $w_{ij}$ is a word obtained from $w$ by deleting all elements except $i$ and $j$.
For example, let $w=312432314$,
\begin{align*}
w_{12}=1221,~~w_{13}=31331,~~w_{14}=1414,~~
w_{23}=32323,~~w_{24}=2424,~~w_{34}=34334,
\end{align*}
then $$\text{\footnotesize Z}(w)=\sum_{i<j}\text{\footnotesize MAJ}(w_{ij})=18.$$
Zeilberger and Bressoud \cite{Zeilberger-1985} proved {\footnotesize Z} is Mahonian by induction.
Greene \cite{Greene-1988} presented  a combinatorial proof.
Han \cite{Han-1992} gave another combinatorial proof by exhibiting a Foata-style bijection, which we will denote as $H_{\text{\tiny Z}}$.
The goal of this section is to establish the equidistribution of the statistics  {\footnotesize MAJ} and {\footnotesize Z} on $\mathfrak{S}_{M}^{\tau}$ for consecutive $\tau$
by proving that Han's bijection $H_{\text{\tiny Z}}$ preserves the consecutive tail permutation.

Before stating  the bijection $H_{\text{\tiny Z}}$,
we need some notions, see \cite{Han-1992}.
Recall that, throughout this paper we let $M=\{1^{k_{1}},2^{k_{2}},\ldots,m^{k_{m}}\}$ with $k_{i}\geq1$ for all $i$.
In this section,
we let $m$ be a fixed number and let
$\textbf{m}=(k_{1},k_{2},\ldots,k_{m}):=\{1^{k_{1}},2^{k_{2}},\ldots,m^{k_{m}}\}$
with $k_{i}\geq0$ for all $i$.
That is,
for the multiset $M$ we assume that $k_{i}\geq1$ and
for the multiset $\textbf{m}$ we assume that $k_{i}\geq0$.

Let $w=x_{1}x_{2}\ldots x_{n}\in\mathfrak{S}_{\textbf{m}}$ and let $x$ be a positive integer,
define
\begin{align*}
C^{x}(w)=y_{1}y_{2}\ldots y_{n},~~~~\text{where}~~y_{i}=C^{x}(x_{i}):=\left
\{
\begin{aligned}
&x_{i}-x, \quad\quad\quad\quad\text{if~}x_{i}>x, \\
&x_{i}-x+m,\quad\quad\text{if~}x_{i}\leq x,
\end{aligned}
\right.
\end{align*}
and
\begin{align*}
C_{x}(w)=z_{1}z_{2}\ldots z_{n},~~~~\text{where}~~z_{i}=C_{x}(x_{i}):=\left
\{
\begin{aligned}
&x_{i}, \quad\quad\quad\quad\text{if~}x_{i}<x, \\
&x_{i}-1, \quad\quad\text{~if~}x_{i}>x, \\
&m, \quad\quad\quad\quad\text{if~}x_{i}=x.
\end{aligned}
\right.
\end{align*}
Note that
$C^{x}(w)\in\mathfrak{S}_{\textbf{m}^{x}}$ and $C_{x}(w)\in\mathfrak{S}_{\textbf{m}_{x}}$,
where
\begin{align}
\textbf{m}^{x}&=(k_{x+1},k_{x+2},\ldots,k_{m},k_{1},k_{2},\ldots,k_{x-1},k_{x}),\label{m^x}\\
\textbf{m}_{x}&=(k_{1},k_{2},\ldots,k_{x-1},k_{x+1},k_{x+2},\ldots,k_{m},k_{x}).\label{m_x}
\end{align}

 \vskip -4mm
Let us  recall the construction of a bijection $\theta_{\textbf{m},\textbf{m}^{\prime}}:\mathfrak{S}_{\textbf{m}}\rightarrow\mathfrak{S}_{\textbf{m}^{\prime}}$,
fixing the statistic {\footnotesize MAJ}.
It is enough to give this construction when $\textbf{m}$ and $\textbf{m}^{\prime}$ differ only by two consecutive letters, say $i$ and $i+1$, that is the bijection
$$
\theta_{i}:\mathfrak{S}_{\textbf{m}}\rightarrow\mathfrak{S}_{\textbf{m}^{\prime}},\text{~where~}
\textbf{m}=(k_{1},\ldots,k_{i},k_{i+1},\ldots,k_{m}),
\textbf{m}^{\prime}=(k_{1},\ldots,k_{i+1},k_{i},\ldots,k_{m}).
$$
We do it as follows: let $w\in\mathfrak{S}_{\textbf{m}}$.
We replace all the $(i+1)i$ factors of this word with a special letter ``$\sim$''.
In the word thus obtained,
the maximum factors containing the two letters $i$ and $i+1$ have the form $i^{a}(i+1)^{b}$
$(a\geq0, b\geq0)$.
We then change these factors to $i^{b}(i+1)^{a}$  and replace each ``$\sim$'' by $(i+1)i$,
to obtain the word $\theta_{i}(w)\in\mathfrak{S}_{\textbf{m}^{\prime}}$.
For example, let $w=1112111222215622$ and $i=1$, we have
\begin{align*}
w~&=~111~21~11222~21~5622\\
&\mapsto111\sim~11222\sim~5622 \\
&\mapsto222\sim~11122\sim~5611\\
&\mapsto222~21~11122~21~5611~=~\theta_{1}(w).
\end{align*}
Note that $\text{Des}(w)=\text{Des}(\theta_{i}(w))$ and thus $\theta_{i}$ fixes the statistic {\footnotesize MAJ}.

The bijection $H_{\text{\tiny Z}}$ is defined, for any word $w\in\mathfrak{S}_{\textbf{m}}$ and any letter $x$, by the following composition:
\begin{align*}
H_{\text{\tiny Z}}(wx)=\left(C_{x}^{-1}\circ H_{\text{\tiny Z}}\circ\theta_{\textbf{m}^{x},\textbf{m}_{x}}\circ C^{x}(w)\right)x.
\end{align*}
\begin{theorem}[Han \cite{Han-1992}]\label{Han-Z}
$H_{\emph{\tiny Z}}$ is a bijection
satisfying
$$\emph{\footnotesize{MAJ}}(w)=\emph{\footnotesize{Z}}(H_{\emph{\tiny Z}}(w))$$
for any $w\in\mathfrak{S}_{\emph{\textbf{m}}}.$
\end{theorem}
 \vskip -2mm
The main result of this section is the following proposition,
which implies the  equidistribution of the statistics  {\footnotesize MAJ} and {\footnotesize Z} on $\mathfrak{S}_{M}^{\tau}$ for consecutive $\tau$.
\begin{proposition}\label{Prop-Han-Z}
Let $M=\{1^{k_{1}},2^{k_{2}},\ldots,m^{k_{m}}\}$ with $k_{i}\geq1$ for all $i\in[m]$,
and let $\tau$ be a consecutive permutation of $[m]$,
the set $\mathfrak{S}_{M}^{\tau}$ is invariant under Han's bijection $H_{\emph{\tiny Z}}$, that is
$$H_{\emph{\tiny Z}}(\mathfrak{S}_{M}^{\tau})=\mathfrak{S}_{M}^{\tau}.$$
\end{proposition}
In the rest of this section, we prove this proposition.
As we will see, this is a somewhat more difficult task.
We first give some notations and lemmas.

For any $j\geq0$, we denote
$$\theta_{j!}=\theta_{j}\circ\cdots\circ\theta_{2}\circ\theta_{1}\circ\theta_{0},\text{~where~}\theta_{0}=\text{id}.$$
Let $S$ be a set,
we denote $S-1:=\{s-1:s\in S\}$.
In general, we denote $S-i:=\{s-i:s\in S\}$.

\begin{lemma}\label{Lemma-c_i-belongs}
Given $w=w_{1}w_{2}\ldots w_{n}\in\mathfrak{S}_{\emph{\textbf{m}}}$,
assume that $\theta_{(m-2)!}(w)=c_{1}c_{2}\ldots c_{n}$.
Let $i\in[n]$,
if $w_{i}=m$, then $c_{i}=m$;
if $w_{i}<m$, then $c_{i}\in R_{i}-1$,
where $R_{i}=\{w_{i-1},w_{i},\ldots,w_{n},m\}$,
and we assume that $w_{0}=m$.
\end{lemma}
\begin{proof}
Let
$$\theta_{j!}(w)=w_{1}^{(j)}w_{2}^{(j)}\ldots w_{n}^{(j)},$$
for $0\leq j\leq m-2$.
So $w_{i}=w_{i}^{(0)}$ and $c_{i}=w_{i}^{(m-2)}$, $1\leq i\leq n$.
Let $w_{0}^{(j)}=w_{n+1}^{(j)}=m$ for all $j$.
Given an index $i\in[n]$, we assume that $w_{i}=a$.
Obviously,
$$w_{i}^{(0)}=w_{i}^{(1)}=w_{i}^{(2)}=\cdots=w_{i}^{(a-2)}=a.$$
Then when $w_{i}=a=m$, we have $c_{i}=w_{i}^{(m-2)}=m$.
Below we assume that  $w_{i}=a<m$.
Then  $a-1\leq m-2$. Consider the word $\theta_{(a-1)!}(w)$.
If $a>1$ and $w_{i}^{(a-1)}=a-1$, then
$$w_{i}^{(a)}=w_{i}^{(a+1)}=\cdots=w_{i}^{(m-2)}=a-1.$$
So $c_{i}=w_{i}^{(m-2)}=a-1=w_{i}-1\in R_{i}-1$.
Below we assume that $w_{i}^{(a-1)}\neq a-1$,
so $w_{i}^{(a-1)}=a$ (including the case of $a=1$).
Assume that $w_{i-1}^{(a-1)}=p$,
and that
\begin{align*}
w_{i}^{(a-1)}=w_{i+1}^{(a-1)}=\cdots=w_{j-1}^{(a-1)}=a,\text{~and~}w_{j}^{(a-1)}=q\neq a,
\end{align*}
where $1\leq i<j\leq n+1$.
The above assumption means that
\begin{align*}
w_{i-1}^{(a-1)}w_{i}^{(a-1)}w_{i+1}^{(a-1)}\ldots w_{j-1}^{(a-1)}w_{j}^{(a-1)}=paa\ldots aq, ~a\neq q.
\end{align*}
We distinguish two cases.

\noindent\textbf{Case 1:~}$p\leq a$.
We further distinguish two subcases.
\begin{itemize}
\item[]
\textbf{Subcase 1.1:~} $a>q$.
It is not hard to see that
$$w_{i}^{(a)}=a+1,~w_{i}^{(a+1)}=a+2,\ldots,~w_{i}^{(m-2)}=m-1.$$
Then $c_{i}=w_{i}^{(m-2)}=m-1\in R_{i}-1$.
\item[]
\textbf{Subcase 1.2:~} $a<q$.
Since $w_{j}^{(a-1)}=q>a$, we have $w_{j}=q$.
It is not hard to see that
\begin{align*}
w_{i}^{(a)}w_{i+1}^{(a)}\ldots w_{j-1}^{(a)}w_{j}^{(a)}&=(a+1)(a+1)\ldots(a+1)q,\\
w_{i}^{(a+1)}w_{i+1}^{(a+1)}\ldots w_{j-1}^{(a+1)}w_{j}^{(a+1)}&=(a+2)(a+2)\ldots(a+2)q,\\
&~~\vdots\\
w_{i}^{(q-2)}w_{i+1}^{(q-2)}\ldots w_{j-1}^{(q-2)}w_{j}^{(q-2)}&=(q-1)(q-1)\ldots(q-1)q.
\end{align*}
If $q=m$, then $c_{i}=w_{i}^{(m-2)}=w_{i}^{(q-2)}=m-1\in
R_{i}-1$. If $q<m$, then $q-1\leq m-2$.
Since $w_{i}^{(q-2)}=q-1$, then $w_{i}^{(q-1)}$ is either $q-1$ or $q$.
If $w_{i}^{(q-1)}=q-1$, then
$$w_{i}^{(q)}=w_{i}^{(q+1)}=\cdots=w_{i}^{(m-2)}=q-1.$$
We have $c_{i}=w_{i}^{(m-2)}=q-1=w_{j}-1\in  R_{i}-1$.
If $w_{i}^{(q-1)}=q$.
Clearly,
$$w_{i}^{(q-1)}w_{i+1}^{(q-1)}\ldots w_{j-1}^{(q-1)}w_{j}^{(q-1)}=qq\ldots q.$$
We now assume that
\begin{align*}
w_{i-1}^{(q-1)}w_{i}^{(q-1)}w_{i+1}^{(q-1)}\ldots w_{k}^{(q-1)}=sqq\ldots qr,~q\neq r,
\end{align*}
where $1\leq i<j<k\leq n+1$.
As $p\leq a$, i.e., $w_{i-1}^{(a-1)}\leq w_{i}^{(a-1)}$, and $\theta_{l}$ fixes the decent set for all $l$, we have $s\leq q$.
Then we return to \textbf{Case 1} and we can give the proof inductively.
\end{itemize}
\noindent\textbf{Case 2:~}$p>a$. We further consider two subcases.
\begin{itemize}
\item[]
\textbf{Subcase 2.1:~} $a>q$.
Since $w_{i-1}^{(a-1)}=p>a$, we see that $w_{i-1}=p$.
It is clear that
$$w_{i}^{(a)}=a+1,~~w_{i}^{(a+1)}=a+2,~\ldots,~~w_{i}^{(p-2)}=w_{i}^{(p-1)}=\cdots=w_{i}^{(m-2)}=p-1.$$
Then $c_{i}=w_{i}^{(m-2)}=p-1=w_{i-1}-1\in R_{i}-1$.
\item[]
\textbf{Subcase 2.2:~} $a<q$.
Combining the arguments of \textbf{Subcase 1.2} and \textbf{Subcase 2.1} gives the proof.
\end{itemize}
We complete the proof.
\end{proof}

Given a word $w=w_{1}w_{2}\ldots w_{n}$  and a set $A$, if $w\cap A\neq\emptyset$, that is
$\{w_{1},w_{2},\ldots, w_{n}\}\cap A\neq\emptyset$,
we denote by $\text{Last}_{A}(w)$ the rightmost letter in $w$ that belongs to $A$.
Obvious that $\text{Last}_{A}(w)\in A$. For example, let $A=\{2,3,5\}$, then
$\text{Last}_{A}(123453441)=3$ and $\text{Last}_{A}(123435441)=5$.
\begin{lemma}\label{Lemma-LastA-theta}
Given a set $A$ with $1<\min(A)\leq\max(A)<m$.
Let $w\in\mathfrak{S}_{\emph{\textbf{m}}}$,
if $w\cap A\neq\emptyset$,
then
$$\emph{Last}_{A-1}\left(\theta_{(m-2)!}(w)\right)=\emph{Last}_{A}(w)-1.$$
\end{lemma}
\begin{proof}
Assume that $w=w_{1}w_{2}\ldots w_{n}$
and $\theta_{(m-2)!}(w)=c_{1}c_{2}\ldots c_{n}$.
Let $B=[m+1]-A$, clearly $(B-1)\cap(A-1)=\emptyset.$
Since $\max(A)<m$, we have $m,m+1\in B$.
Let $w_{p}$ be the rightmost letter in $w$ that belongs to $A$.
Assume that $w_{p}=a$.
Since $a\in A$, we have $2\leq a\leq m-1$.
It is clear that $\text{Last}_{A}(w)=w_{p}=a\in A$ and
$w_{p+1},w_{p+2},\ldots, w_{n}\in B$.
Then
$$E:=\{w_{p+1},w_{p+2},\ldots, w_{n},m,m+1\}\subseteq B.$$
Given an index $i$ with $p+2\leq i\leq n$,
by Lemma \ref{Lemma-c_i-belongs} we have
\begin{align}\label{C_i_in_B_1}
c_{i}\in\{w_{i-1}-1,w_{i}-1,\ldots,w_{n}-1,m-1,m\}\subseteq E-1\subseteq B-1.
\end{align}
For $i=p+1$, by Lemma \ref{Lemma-c_i-belongs} we have
\begin{align}\label{C_i_in_B_2}
c_{p+1}\in\{w_{p}-1\}\cup(E-1)\subseteq\{a-1\}\cup(B-1).
\end{align}
Since $(B-1)\cap(A-1)=\emptyset$, by  (\ref{C_i_in_B_2}) and (\ref{C_i_in_B_1})
we have
\begin{align}
&c_{p+1}=a-1\in A-1 \text{~or~}c_{p+1}\notin A-1,\label{C_i_in_B_3}\\
&c_{p+2},c_{p+3},\ldots,c_{n}\notin A-1.\label{C_i_in_B_4}
\end{align}

Assume that $\theta_{(a-2)!}(w)=e=e_{1}e_{2}\ldots e_{n}$ (note that $a\geq2$).
Clearly $e_{p}=w_{p}=a$.
We now apply the operator $\theta_{a-1}$ to $e$.
If $e_{p}$ becomes $a-1$,
since $a-1$ is unchanged after we apply the operator $\theta_{m-2}\circ\cdots\circ\theta_{a+1}\circ\theta_{a}$,
then  $c_{p}=a-1$,
combining this with (\ref{C_i_in_B_3}) and (\ref{C_i_in_B_4})
we get
$$\text{Last}_{A-1}\left(c_{1}c_{2}\ldots c_{n}\right)=a-1=w_{p}-1=\text{Last}_{A}(w)-1,$$
we complete the proof for this case.
Below we assume that $e_{p}$ is unchanged after we apply the operator $\theta_{a-1}$ to $e$.
We consider two cases.

\noindent\textbf{Case 1:~}$e_{p}e_{p+1}=a(a-1)$.
It is clear that $e_{p+1}$ is unchanged after we apply the operator
$\theta_{m-2}\circ\cdots\circ\theta_{a}\circ\theta_{a-1}$ to $e$.
Therefore, $c_{p+1}=a-1$.
Combining this with (\ref{C_i_in_B_4}),
we have
$$\text{Last}_{A-1}\left(c_{1}c_{2}\ldots c_{n}\right)=c_{p+1}=a-1=w_{p}-1=\text{Last}_{A}(w)-1.$$

\noindent\textbf{Case 2:~}$e_{p+1}\neq a-1$.
Let $\theta_{a-1}(e)=\theta_{(a-1)!}(w)=h=h_{1}h_{2}\ldots h_{n}$.
By our assumption that $e_{p}$ is unchanged after we apply the operator $\theta_{a-1}$ to $e$,
we have $h_{p}=e_{p}=a$.
Because $h_{p}=e_{p}=a$,
there must be an index $q$ with $q<p$ such that $e_{q}e_{q+1}\ldots e_{p}=(a-1)\ldots(a-1)a\ldots a$,
and
$h_{q}h_{q+1}\ldots h_{p}$ also has the form $(a-1)\ldots(a-1)a\ldots a$.
Assume that $h_{s}=a-1$ and $h_{s+1}=\cdots=h_{p}=a$, where $q\leq s< p$.
Since $a-1$ is unchanged after we apply the operator $\theta_{m-2}\circ\cdots\circ\theta_{a+1}\circ\theta_{a}$ to $h$,
then  $c_{s}=h_{s}=a-1$.
Because $h=\theta_{(a-1)!}(w)$ and $h_{s}h_{s+1}\ldots h_{p}=(a-1)a\ldots a$,
we see that
\begin{align}\label{EQ-w_{s+1}}
F:=\{w_{s},w_{s+1},\ldots,w_{p}\}\subseteq\{1,2,\ldots,a\}.
\end{align}
Note that  applying  the operator $\theta_{m-2}\circ\cdots\circ\theta_{a+1}\circ\theta_{a}$ to $h$ will not decrease any letter $a$ in $h$.
Since $h_{s+1}=h_{s+2}=\cdots=h_{p}=a$, we have
\begin{align}\label{EQ-c_{s+1}...<=a}
c_{s+1},c_{s+2},\ldots,c_{p}\geq a.
\end{align}
Given an index $i$ with $s+1\leq i\leq p$, by Lemma \ref{Lemma-c_i-belongs},
we have
\begin{align*}
c_{i}\in\{w_{i-1}-1,w_{i}-1,\ldots,w_{n}-1,m-1,m\}\subseteq (F-1)\cup(E-1).
\end{align*}
By (\ref{EQ-w_{s+1}}) and (\ref{EQ-c_{s+1}...<=a})
we have $c_{i}\notin F-1$,
then $c_{i}\in E-1\subseteq B-1$, where $s+1\leq i\leq p$.
Since $(B-1)\cap(A-1)=\emptyset$, we have
\begin{align}\label{c_s+1...c_p}
c_{s+1},c_{s+2},\ldots,c_{p}\notin A-1.
\end{align}
Combining (\ref{c_s+1...c_p}), (\ref{C_i_in_B_3}), (\ref{C_i_in_B_4}) and the fact that $c_{s}=a-1\in A-1$,
we get
$$\text{Last}_{A-1}\left(c_{1}c_{2}\ldots c_{n}\right)=a-1=w_{p}-1=\text{Last}_{A}(w)-1.$$
We complete the proof.
\end{proof}

Comparing equations (\ref{m^x}) and (\ref{m_x}), we see that
$$\theta_{\textbf{m}^{x},\textbf{m}_{x}}=\theta_{(m-2)!}^{m-x}:=\underbrace{\theta_{(m-2)!}\circ\theta_{(m-2)!}\circ\cdots\circ\theta_{(m-2)!}}_{m-x}.$$
We define the map $\phi_{x}:=\theta_{\textbf{m}^{x},\textbf{m}_{x}}\circ C^{x}=\theta_{(m-2)!}^{m-x}\circ C^{x}$.
Then
\begin{align*}
H_{\text{\tiny Z}}(wx)=\left(C_{x}^{-1}\circ H_{\text{\tiny Z}}\circ\phi_{x}(w)\right)x.
\end{align*}
\begin{lemma}\label{Lemma-LastA}
Given a set $A$ and a number $x$ with $1\leq\min(A)\leq\max(A)<x\leq m$.
Let $w\in\mathfrak{S}_{\emph{\textbf{m}}}$,
if $w\cap A\neq\emptyset$,
then
$$\emph{Last}_{A}(\phi_{x}(w))=\emph{Last}_{A}(w).$$
\end{lemma}
\begin{proof}
Let $m-x=d$. Since $\max(A)<x$, by the definition of $C^{x}$ we see that
\begin{align}\label{eq-C^x}
\text{Last}_{A+d}(C^{x}(w))=\text{Last}_{A}(w)+d.
\end{align}
Since $\max(A)<x$, we have $\max(A)+d<m$.
Thus, $1<\min(A+j)\leq\max(A+j)<m$ for any $1\leq j\leq d$.
By Lemma \ref{Lemma-LastA-theta}, we have
\begin{align*}
\text{Last}_{A}\left(\theta_{(m-2)!}^{d}\circ C^{x}(w)\right)
&=\text{Last}_{A+1}\left(\theta_{(m-2)!}^{d-1}\circ C^{x}(w)\right)-1\\
&=\text{Last}_{A+2}\left(\theta_{(m-2)!}^{d-2}\circ C^{x}(w)\right)-2\\
&~~~\vdots\\
&=\text{Last}_{A+d-1}\left(\theta_{(m-2)!}^{1}\circ C^{x}(w)\right)-(d-1)\\
&=\text{Last}_{A+d}\left(C^{x}(w)\right)-d.
\end{align*}
Combining this with (\ref{eq-C^x}) yields
\begin{align*}
\text{Last}_{A}\left(\phi_{x}(w)\right)=\text{Last}_{A}(w),
\end{align*}
completing the proof.
\end{proof}

\noindent\emph{{Proof of Proposition \ref{Prop-Han-Z}.}}
We prove the following stronger property.
Given $w\in\mathfrak{S}_{\textbf{m}}$.
For any $l$ with $1\leq l\leq m$,
assume that $\tau=\tau_{1}\tau_{2}\ldots\tau_{l}$ is a consecutive permutation of $[l]$.
Denote
$$A_{r}=\{\tau_{1},\tau_{2},\ldots,\tau_{r}\},~1\leq r\leq l.$$
If
\begin{align}\label{eq-stronger-if}
\text{Last}_{A_{r}}\left(w\right)=\tau_{r}, \text{~for~any~} 1\leq r\leq l,
\end{align}
then
\begin{align}\label{eq-stronger-then}
\text{Last}_{A_{r}}\left(H_{\text{\tiny Z}}(w)\right)=\tau_{r}, \text{~for~any~} 1\leq r\leq l.
\end{align}
To see the above property implies Proposition \ref{Prop-Han-Z},
note that taking $k_{i}\geq1$ for $1\leq i\leq m$ in $\textbf{m}=(k_{1},k_{2},\ldots,k_{m})$, we have $\textbf{m}=M$.
Then taking $l=m$, (\ref{eq-stronger-if}) means that $w\in\mathfrak{S}_{M}^{\tau}$,
and (\ref{eq-stronger-then}) means that $H_{\text{\tiny Z}}(w)\in\mathfrak{S}_{M}^{\tau}$.
Thus, the above property implies Proposition \ref{Prop-Han-Z}.

Below we use induction on $n$, which is the length of $w$, to prove the above  property.
The initial case of $n=1$ is obvious.
Assume the above  property is true for $n-1$ and prove it for $n$.
Let $w=w_{1}w_{2}\ldots w_{n}$
and $w^{\prime}=w_{1}w_{2}\ldots w_{n-1}$.
By definition we see that
\begin{align}\label{Eq-Hz-wn}
H_{\text{\tiny Z}}(w)=\left(C_{w_{n}}^{-1}\circ H_{\text{\tiny Z}}\circ\phi_{w_{n}}(w^{\prime})\right)w_{n}.
\end{align}
We first consider the case that $w_{n}\in A_{l}$.
Then $\text{Last}_{A_{l}}\left(w\right)=w_{n}$.
By (\ref{eq-stronger-if}) we see that $w_{n}=\tau_{l}$.
Since $\tau$ is a consecutive permutation of $[l]$,
then $\tau_{l}=l$ or $\tau_{l}=1$.
We distinguish two cases.

\noindent\textbf{Case 1:}~$\tau_{l}=l$.
From (\ref{Eq-Hz-wn}) we have
\begin{align}\label{eq-case1}
H_{\text{\tiny Z}}(w)=\left(C_{l}^{-1}\circ H_{\text{\tiny Z}}\circ\phi_{l}(w^{\prime})\right)l.
\end{align}
Given $r$ with $1\leq r\leq l-1$.
Obviously, $1\leq\min(A_{r})\leq\max(A_{r})<l\leq m$.
Since $w\cap A_{r}\neq\emptyset$ and $l\notin A_{r}$, then $w^{\prime}\cap A_{r}\neq\emptyset$.
By Lemma \ref{Lemma-LastA}, we get
\begin{align}\label{eq-1}
\text{Last}_{A_{r}}(\phi_{l}(w^{\prime}))=\text{Last}_{A_{r}}(w^{\prime})=\text{Last}_{A_{r}}(w)=\tau_{r}, \text{~for~} 1\leq r\leq l-1.
\end{align}
Note that $\phi_{l}(w^{\prime})$ is of length $n-1$,
we apply the induction hypothesis to $\phi_{l}(w^{\prime})$,
with taking $l^{\prime}=l-1$ and $\tau^{\prime}=\tau_{1}\tau_{2}\ldots\tau_{l-1}$.
Clearly, $\tau^{\prime}$ is a consecutive permutation of $[l-1]$ as $\tau_{l}=l$.
By (\ref{eq-1}), we see that $\phi_{l}(w^{\prime})$ and $\tau^{\prime}$ satisfy the condition (\ref{eq-stronger-if}),
then
\begin{align}\label{eq-2}
\text{Last}_{A_{r}}\left(H_{\text{\tiny Z}}\circ\phi_{l}(w^{\prime})\right)
=\tau_{r}, \text{~for~}1\leq r\leq l-1.
\end{align}
Since $\max(A_{r})<l$, we see that
\begin{align}\label{eq-3}
\text{Last}_{A_{r}}\left(C_{l}^{-1}\circ H_{\text{\tiny Z}}\circ\phi_{l}(w^{\prime})\right)
=\text{Last}_{A_{r}}\left(H_{\text{\tiny Z}}\circ\phi_{l}(w^{\prime})\right)=\tau_{r}, \text{~for~}1\leq r\leq l-1.
\end{align}
Combining (\ref{eq-3}) with (\ref{eq-case1}) yields
\begin{align*}
\text{Last}_{A_{r}}\left(H_{\text{\tiny Z}}(w)\right)=\tau_{r},
\text{~for~}1\leq r\leq l-1.
\end{align*}
Since $A_{l}=\{1,2,\ldots,l\}$, by (\ref{eq-case1}) we have
\begin{align*}
\text{Last}_{A_{l}}\left(H_{\text{\tiny Z}}(w)\right)=l=\tau_{l}.
\end{align*}
Thus, (\ref{eq-stronger-then}) holds for $1\leq r\leq l$,
and we complete the proof for this case.

\noindent\textbf{Case 2:}~$\tau_{l}=1$.
In this case we have
\begin{align}\label{eq-case2}
H_{\text{\tiny Z}}(w)=\left(C_{1}^{-1}\circ H_{\text{\tiny Z}}\circ\phi_{1}(w^{\prime})\right)1.
\end{align}
Clearly
$\theta_{\textbf{m}^{1},\textbf{m}_{1}}=\text{id}$,
then $\phi_{1}(w^{\prime})=C^{1}(w^{\prime})=d_{1}d_{2}\ldots d_{n-1}$,
where
\begin{align}\label{Eq-d_i}
d_{i}=\left
\{
\begin{aligned}
&w_{i}-1, \quad\quad\text{if~}w_{i}>1, \\
&m,\quad\quad\quad\quad\text{if~}w_{i}=1.
\end{aligned}
\right.
\end{align}
Let $\tau^{\prime}=\tau_{1}^{\prime}\tau_{2}^{\prime}\ldots\tau_{l-1}^{\prime}$,
where $\tau_{i}^{\prime}=\tau_{i}-1$.
Since $\tau=\tau_{1}\tau_{2}\ldots\tau_{l}$ is a consecutive permutation of $[l]$ and $\tau_{l}=1$,
we see that $\tau^{\prime}$ is a consecutive permutation of $[l-1]$.
For $1\leq r\leq l-1$, let
$$A_{r}^{\prime}=\{\tau_{1}^{\prime},\tau_{2}^{\prime},\ldots,\tau_{r}^{\prime}\}=A_{r}-1.$$
Note that $\min(A_{r})>1$ for $1\leq r\leq l-1$, by (\ref{Eq-d_i}) we see that
\begin{align}\label{eq-4}
\text{Last}_{A_{r}^{\prime}}\left(\phi_{1}(w^{\prime})\right)=
\text{Last}_{A_{r}}\left(w^{\prime}\right)-1=\tau_{r}-1=\tau^{\prime}_{r}, \text{~for~}1\leq r\leq l-1.
\end{align}
Applying the induction hypothesis to $\phi_{1}(w^{\prime})$,
then
\begin{align}\label{eq-5}
\text{Last}_{A_{r}^{\prime}}\left(H_{\text{\tiny Z}}\circ\phi_{1}(w^{\prime})\right)
=\tau^{\prime}_{r}, \text{~for~}1\leq r\leq l-1.
\end{align}
It is not hard to see that
\begin{align}\label{eq-6}
\text{Last}_{A_{r}}\left(C_{1}^{-1}\circ H_{\text{\tiny Z}}\circ\phi_{1}(w^{\prime})\right)
=\text{Last}_{A_{r}^{\prime}}\left(H_{\text{\tiny Z}}\circ\phi_{1}(w^{\prime})\right)+1.
\end{align}
Then we obtain
\begin{align}\label{eq-6}
\text{Last}_{A_{r}}\left(C_{1}^{-1}\circ H_{\text{\tiny Z}}\circ\phi_{1}(w^{\prime})\right)
=\tau^{\prime}_{r}+1=\tau_{r}, \text{~for~}1\leq r\leq l-1.
\end{align}
Note that $1\notin A_{r}$ for $1\leq r\leq l-1$,
combining (\ref{eq-6}) with (\ref{eq-case2}), we have
\begin{align*}
\text{Last}_{A_{r}}\left(H_{\text{\tiny Z}}(w)\right)=\tau_{r}, \text{~for~}1\leq r\leq l-1.
\end{align*}
Clearly, $A_{l}=\{1,2,\ldots,l\}$, by (\ref{eq-case2}) we have
\begin{align*}
\text{Last}_{A_{l}}\left(H_{\text{\tiny Z}}(w)\right)=1=\tau_{l}.
\end{align*}
Thus, (\ref{eq-stronger-then}) holds for $1\leq r\leq l$,
completing the proof  for this case.

We now consider the case that $w_{n}\notin A_{l}$.
Since $A_{l}=[l]$ and $w_{n}\notin A_{l}$, we have $l<w_{n}\leq m$.
Then for any $r$, $1\leq r\leq l$, we have $1\leq\min(A_{r})\leq\max(A_{r})<w_{n}\leq m$.
The proof for this case is very similar to that of \textbf{Case 1}, and we omit it.
\hfill{$\mbox{\rule[0pt]{0.7ex}{1.3ex}}$}

\section{Equidistribution of (mstc,{\large INV}) and (des,{\large MAJ}) on $\mathcal{P}_{M}$}\label{section-bi-1}
The permutation statistic stc was introduced by Skandera \cite{Skandera-2001},
he proved that (stc, {\footnotesize INV}) is Euler-Mahonian.
Carnevale \cite{Carnevale-2017} extended  the  statistic stc to words, which he called mstc,
and proved that (mstc, {\footnotesize INV}) is Euler-Mahonian,
that is (mstc, {\footnotesize INV}) and (des, {\footnotesize MAJ})  are equidistributed on words.
In this section,
we prove that  (mstc, {\footnotesize INV}) and (des, {\footnotesize MAJ}) are equidistributed when restricted to $\mathcal{P}_{M}$.
Here we give an equivalent definition of the statistic mstc.

First, we define the map $\text{std}:\mathfrak{S}_{M}\rightarrow \mathfrak{S}_{n}$.
Given $w=w_{1}w_{2}\ldots w_{n}\in\mathfrak{S}_{M}$,
define std$(w)$ to be the permutation $\pi_{1}\pi_{2}\ldots \pi_{n}\in\mathfrak{S}_{n}$  such that
$\pi_{i}<\pi_{j}$ if and only if either $w_{i}<w_{j}$ or $w_{i}=w_{j}$ with $i<j$.
For example,
std$(32112133)=64125378$.
Clearly, $\text{\footnotesize INV}(\text{std}(w))=\text{\footnotesize INV}(w)$.

Second, we define the map $I:\mathfrak{S}_{n}\rightarrow \text{\textbf{I}}_{n}$, where
$\text{\textbf{I}}_{n}:=\{(c_{1},c_{2},\ldots,c_{n}): 0\leq c_{i}\leq i-1\}.$
Given $\pi\in\mathfrak{S}_{n}$, for $i\in[n]$,
let $c_{i}$  be  the number of letters $j$ to the right of the letter $i$ in $\pi$ such that $j<i$.
Then define
$$I(\pi)=(c_{1},c_{2},\ldots,c_{n}).$$
Clearly, {\footnotesize INV}$(\pi)=\sum_{i=1}^{n}c_{i}$.
For example, let $\pi=64125378$, then $I(\pi)=(0,0,0,3,1,5,0,0)$ and {\footnotesize INV}$(\pi)=3+1+5=9$.

Third, we give the definition of the statistic eul  on $\text{\textbf{I}}_{n}$,
see \cite{Han-1990-1}.
Let $c=(c_{1},c_{2},\ldots,c_{n})\in \text{\textbf{I}}_{n}$.
If $n=1$, define $\text{eul}(c)=0$;
if $n\geq2$, let $c^{\prime}=(c_{1},c_{2},\ldots, c_{n-1})$,
then define
\begin{equation*}
\text{eul}(c)=\left\{
\begin{aligned}
&\text{eul}(c^{\prime}), \quad\quad\text{if~}c_{n}\leq \text{eul}(c^{\prime}), \\
&\text{eul}(c^{\prime})+1, ~\text{if~}c_{n}> \text{eul}(c^{\prime}).
\end{aligned}
\right.
\end{equation*}

Finally, for $w\in\mathfrak{S}_{M}$,
define
\begin{align}\label{definition-mstc-1}
\text{mstc}(w)=\text{eul}\circ I\circ\text{std}(w).
\end{align}

\begin{remark}
\emph{Let $\textbf{\text{i}}(\pi)=\pi^{-1}$. Originally, Carnevale \cite{Carnevale-2017} defined the statistic mstc to be
\begin{align}\label{definition-mstc-2}
\text{mstc}(w)=\text{stc}\circ\textbf{\text{i}}\circ\text{std}(w).
\end{align}
In \cite{Liu-2021} (at the end of Section 7),
the author proved that for any permutation $\pi$,
$$\text{eul}\circ I(\pi)=\text{stc}\circ\textbf{\text{i}}(\pi).$$
Then
$$\text{eul}\circ I\circ \text{std}(w)=\text{stc}\circ\textbf{\text{i}}\circ \text{std}(w),$$
which shows the equivalence of (\ref{definition-mstc-1}) and (\ref{definition-mstc-2}).}
\end{remark}

We now review the bijection $\Psi$ on permutations that takes {\footnotesize INV} to {\footnotesize MAJ}, which is essentially due to Carlitz \cite{Carlitz-1975}.
Also see \cite{Remmel-2015}.
See \cite{Liu-2021} for a $k$-extension of this bijection.

We will recursively define the bijection $\Psi$.
Given a permutation $\pi\in\mathfrak{S}_{n}$, assume that $I(\pi)=(c_{1},c_{2},\ldots,c_{n})$.
Let $\pi^{\prime}\in\mathfrak{S}_{n-1}$ be the permutation obtained from $\pi$ by deleting the letter $n$.
Assume that $\Psi(\pi^{\prime})\in\mathfrak{S}_{n-1}$ has been defined,
we will define $\Psi(\pi)$  from $\Psi(\pi^{\prime})$  by inserting the letter $n$.
There are $n$ positions where we can insert $n$ in $\Psi(\pi^{\prime})$ to obtain a permutation $\Psi(\pi)$.
We first label the positions of $\Psi(\pi^{\prime})$ according to the following rules:
\begin{itemize}
\item[(a)] Label the  position after $\Psi(\pi^{\prime})$ with 0.
\item[(b)] Label the positions following the descents of $\Psi(\pi^{\prime})$ from right to left with $1, 2, \ldots, \text{des}(\Psi(\pi^{\prime}))$.
\item[(c)] Label the remaining  positions from left to right with $\text{des}(\Psi(\pi^{\prime}))+1, \ldots, n-1$.
\end{itemize}
Finally, we insert the letter $n$ into the position labeled by $c_{n}$,
let $\Psi(\pi)$ be the resulting permutation.

For example,
if $\Psi(\pi^{\prime})=64125378$,
we can write the labeling of the positions of $\Psi(\pi^{\prime})$ as subscripts to get
$$_{4}6_{3}4_{2}1_{5}2_{6}5_{1}3_{7}7_{8}8_{0}.$$
If $c_{9}=6$, then $\Psi(\pi)=641295378$.

Given a multiset $M=\{1^{k_{1}},2^{k_{2}},\ldots,m^{k_{m}}\}$ with $|M|=n$,
we extend $\Psi$ to $\mathfrak{S}_{M}$, which we denote $\Psi_{M}$.
First, we define a map $\text{istd}_{M}:\mathfrak{S}_{n}\rightarrow\mathfrak{S}_{M}$.
Given $\pi\in\mathfrak{S}_{n}$,
$\text{istd}_{M}(\pi)\in\mathfrak{S}_{M}$ is obtained from $\pi$ by
replacing the letters $1,2,\ldots,k_{1}$ with $1$,
replacing the letters $k_{1}+1,k_{1}+2,\ldots,k_{1}+k_{2}$ with $2$,
\ldots,
replacing the letters $1+\sum_{i=1}^{m-1}k_{i},\ldots,\sum_{i=1}^{m}k_{i}$ with $m$.
Clearly for any $w\in\mathfrak{S}_{M}$, we have
$$\text{istd}_{M}\circ\text{std}(w)=w.$$
For $w\in\mathfrak{S}_{M}$,
we define
\begin{align}\label{definition Psi M}
\Psi_{M}(w)=\text{istd}_{M}\circ\Psi\circ\text{std}(w).
\end{align}
See \cite{Wilson-2016} (Section 2.2) for another equivalent description of $\Psi_{M}$.
\begin{example}
\emph{Let $M=\{1^{2},2^{2},3^{2}\}$ and let $w=213123\in\mathfrak{S}_{M}$.
So $\pi=\text{std}(w)=315246$. Clearly, $I(\pi)=(c_{1},c_{2},c_{3},c_{4},c_{5},c_{6})=(0,0,2,0,2,0)$.
We give  the procedure for creating $\Psi(\pi)$:
\begin{align*}
_{1}1_{0}~\xrightarrow{c_{2}=0}~_{1}1_{2}2_{0}~\xrightarrow{c_{3}=2}~_{2}1_{3}3_{1}2_{0}~\xrightarrow{c_{4}=0}~_{2}1_{3}3_{1}2_{4}4_{0}
~\xrightarrow{c_{5}=2}~_{3}5_{2}1_{4}3_{1}2_{5}4_{0}~\xrightarrow{c_{6}=0}~513246=\Psi(\pi).
\end{align*}
Then $\Psi_{M}(w)=\text{istd}_{M}\circ\Psi(\pi)=\text{istd}_{M}(513246)=312123$.}
\end{example}

Our next goal is to prove the following property of $\Psi_{M}$.
\begin{proposition}\label{Prop Psi_M}
Let $M=\{1^{k_{1}},2^{k_{2}},\ldots,m^{k_{m}}\}$ with $k_{i}\geq1$ for all $i\in[m]$.
For any $w\in\mathfrak{S}_{M}$, we have
$$
(\emph{mstc}, \emph{\footnotesize INV})w=(\emph{des}, \emph{\footnotesize MAJ})\Psi_{M}(w).
$$
\end{proposition}
To prove Proposition \ref{Prop Psi_M}, we need some lemmas.
\begin{lemma}\label{Prop-Liu-bi-statistic}
Given $\pi\in\mathfrak{S}_{n}$, let $I(\pi)=(c_{1},c_{2},\ldots,c_{n})$.
For any $1\leq i\leq n-1$,
if $c_{i}\geq c_{i+1}$,
then in permutation $\Psi(\pi)$, the letter $i+1$ is not immediately followed by the letter $i$.
\end{lemma}
\begin{proof}
Let $\pi_{(i)}$ be the subpermutation of $\pi$ by deleting the letters $i+1,i+2,\ldots,n$,
and let $\gamma_{i}=\Psi(\pi_{(i)})$. Clearly $\gamma_{n}=\Psi(\pi)$.
Consider $\gamma_{i-1}$,
assume that we have labeled the positions of $\gamma_{i-1}$ by $0,1,\ldots,i-1$.
Assume that des$(\gamma_{i-1})=d$.
We now insert the letters $i$ and $i+1$ (in order) into $\gamma_{i-1}$  to obtain $\gamma_{i}$ and $\gamma_{i+1}$.
If $c_{i}\leq d$, then $c_{i+1}\leq c_{i}\leq d$,
by rules (a) and (b) we see that
the letter $i+1$ is on the right of the letter $i$ in  $\gamma_{i+1}$,
thus, the letter $i+1$ is not immediately followed by $i$ in $\gamma_{i+1}$,
as well as in $\gamma_{n}$.
If $c_{i}>d$
and the letter $i+1$ is immediately followed by $i$ in $\gamma_{i+1}$,
by rule (c),
it is not hard to see that $c_{i+1}=c_{i}+1$,
contradicting the assumption that $c_{i}\geq c_{i+1}$,
and we complete the proof.
\end{proof}

By a similar argument, we can obtain the following lemma,
and we omit the proof of it.

\begin{lemma}\label{Prop-Liu-bi-statistic-s}
Given $\pi\in\mathfrak{S}_{n}$, let $I(\pi)=(c_{1},c_{2},\ldots,c_{n})$.
For any $1\leq i\leq n-1$, and $1\leq s\leq n-i$,
if $c_{i}\geq c_{i+1}\geq\cdots\geq c_{i+s} $,
then in permutation $\Psi(\pi)$, the letter $i+s$ is not immediately followed by the letter $i$.
\end{lemma}

\begin{lemma}\label{lemma istd_M}
Let $M=\{1^{k_{1}},2^{k_{2}},\ldots,m^{k_{m}}\}$ with $k_{i}\geq1$ for all $i\in[m]$.
Let $w\in\mathfrak{S}_{M}$, then
\begin{align*}
\emph{Des}(\emph{istd}_{M}\circ\Psi\circ\emph{std}(w))=\emph{Des}(\Psi\circ\emph{std}(w)).
\end{align*}
\end{lemma}
\begin{proof}
Assume that $\pi=\text{std}(w)\in\mathfrak{S}_{n}$.
Given $i\in[m]$,
assume that
$w_{i_{1}}=w_{i_{2}}=\cdots=w_{i_{k_{i}}}=i$, where $i_{1}<i_{2}<\cdots<i_{k_{i}}$.
Let $a=\sum_{j=1}^{i-1}k_{j}$,
then $\pi_{i_{1}}=a+1, \pi_{i_{2}}=a+2, \ldots, \pi_{i_{k_{i}}}=a+k_{i}$.
Assume that $I(\pi)=(c_{1},c_{2},\ldots,c_{n})$.
It is not hard to see that
$$c_{a+1}\geq c_{a+2}\geq\cdots\geq c_{a+k_{i}}.$$
To obtain $\text{istd}_{M}\circ\Psi(\pi)$,
we need replace the letters $a+1,a+2,\ldots,a+k_{i}$ in $\Psi(\pi)$ with $i$.
By Lemma \ref{Prop-Liu-bi-statistic-s},
we see that in $\Psi(\pi)$,
the letter $a+s_{2}$ is not immediately followed by the letter $a+s_{1}$,
for any $1\leq s_{1}<s_{2}\leq k_{i}$.
Thus, the replacement fixes the descent set, that is,
$$\text{Des}(\text{istd}_{M}\circ\Psi(\pi))=\text{Des}(\Psi(\pi)),$$
which completes the proof.
\end{proof}
\begin{lemma}\label{lemma Psi permutation}
For any permutation $\pi$, we have
$$
(\emph{eul}\circ I, \emph{\footnotesize INV})\pi=(\emph{des}, \emph{\footnotesize MAJ})\Psi(\pi).
$$
\end{lemma}

The proof of the lemma can be found in the proof of Theorem 7.1 in \cite{Liu-2021} (the case $k=1$).

\noindent\emph{Proof of Proposition \ref{Prop Psi_M}.}
Given $w\in \mathfrak{S}_{M}$, we have
\begin{align*}
\text{mstc}(w)&=\text{eul}\circ I\circ\text{std}(w)~\quad\quad\quad\quad~(\text{by~}(\ref{definition-mstc-1}))\\
&=\text{des}(\Psi\circ\text{std}(w))~\quad\quad\quad\quad~(\text{by~Lemma~}\ref{lemma Psi permutation})\\
&=\text{des}(\text{istd}_{M}\circ\Psi\circ\text{std}(w))~\quad(\text{by~Lemma~}\ref{lemma istd_M})\\
&=\text{des}(\Psi_{M}(w))~\quad\quad\quad\quad\quad\quad(\text{by~}(\ref{definition Psi M}))
\end{align*}
and
\begin{align*}
\text{{\footnotesize INV}}(w)&=\text{{\footnotesize INV}}(\text{std}(w))\\
&=\text{{\footnotesize MAJ}}(\Psi\circ\text{std}(w))~~\quad\quad\quad\quad~(\text{by~Lemma~}\ref{lemma Psi permutation})\\
&=\text{{\footnotesize MAJ}}(\text{istd}_{M}\circ\Psi\circ\text{std}(w))~\quad~(\text{by~Lemma~}\ref{lemma istd_M})\\
&=\text{{\footnotesize MAJ}}(\Psi_{M}(w)),\quad\quad\quad\quad\quad\quad~(\text{by~}(\ref{definition Psi M}))
\end{align*}
completing the proof.   \hfill{$\mbox{\rule[0pt]{0.7ex}{1.3ex}}$}

The following proposition shows that $\Psi_{M}$ preserves the increasing tail permutation.
\begin{proposition}\label{Prop-Euler-inv}
Let $M=\{1^{k_{1}},2^{k_{2}},\ldots,m^{k_{m}}\}$  with $k_{i}\geq1$ for all $i\in[m]$,
the set $\mathcal{P}_{M}$ is invariant under $\Psi_{M}$, that is
$$\Psi_{M}(\mathcal{P}_{M})=\mathcal{P}_{M}.$$
\end{proposition}
\begin{proof}
Let $w=w_{1}w_{2}\ldots w_{n}\in\mathcal{P}_{M}$,
assume that $w_{s_{j}}$ is the last  occurrence of $j$, $1\leq j\leq m$.
Then $w_{s_{1}}w_{s_{2}}\ldots w_{s_{m}}=12\ldots m$,
where $s_{1}<s_{2}<\cdots<s_{m}$.
Let $\pi=\pi_{1}\pi_{2}\ldots\pi_{n}=\text{std}(w)$.
We denote $a_{j}=\sum_{l=1}^{j}k_{l}$.
Then $\pi_{s_{1}}=a_{1},
\pi_{s_{2}}=a_{2},\ldots, \pi_{s_{m}}=a_{m}$.
Let $I(\pi)=(c_{1},c_{2},\ldots,c_{n})$.
It is not hard to see that $c_{a_{1}}=c_{a_{2}}=\cdots=c_{a_{m}}=0$.
By rule (a) of the construction of $\Psi(\pi)$, we see that in $\Psi(\pi)$,
all the letters to the right of $a_{j}$ are larger than $a_{j}$,
for $1\leq j\leq m$.
This implies that $\text{istd}_{M}\left(\Psi(\pi)\right)\in\mathcal{P}_{M}$,
that is, $\Psi_{M}(w)\in\mathcal{P}_{M}$, completing the proof.
\end{proof}
Combining Proposition  \ref{Prop Psi_M} with Proposition \ref{Prop-Euler-inv} gives the equidistribution of the bi-statistics (mstc, {\footnotesize INV}) and (des, {\footnotesize MAJ})  on $\mathcal{P}_{M}$.

\section{Equidistribution of {\large INV} and $r$-{\large MAJ}  on $\mathcal{P}_{M}$}\label{section-inv-rmaj}
The $r$-major index,
denoted $r\text{-\footnotesize MAJ}$, was introduced by
Rawlings \cite{Rawlings-1981} for  permutations,
and extended  to words in \cite{Rawlings-1981-2}.

First,
define
the \emph{$r$-descent set}, denoted $r$-Des,
and \emph{$r$-inversion set}, denoted $r$-Inv,
for word $w=w_{1}w_{2}\ldots w_{n}$ as follows:
\begin{align*}
r\text{-Des}(w)&=\{i\in\{1,2,\ldots,n-1\}:w_{i}\geq w_{i+1}+r\},\\
r\text{-Inv}(w)&=\{(i,j):1\leq i<j\leq n, w_{i}-r<w_{j}<w_{i}\}.
\end{align*}
Now, $r\text{-\footnotesize MAJ}$ is defined by
\begin{align*}
r\text{-\footnotesize MAJ}(w)=|r\text{-Inv}(w)|+\sum_{i\in r\text{-Des}(w)}i.
\end{align*}
It is clear that $r${-\footnotesize MAJ} reduces to {\footnotesize MAJ} when $r=1$ and to  {\footnotesize INV} when $r\geq n$,
thus the family of $r${-\footnotesize MAJ} interpolates between {\footnotesize{MAJ}} and {\footnotesize{INV}}.

Rawlings \cite{Rawlings-1981} showed that $r$-{\footnotesize MAJ} is equidistributed with {\footnotesize INV} on permutations by constructing a bijection on permutations that takes {\footnotesize INV} to $r$-{\footnotesize MAJ},
which is a generalization of Carlitz's bijection  which we have given in the previous section.
In \cite{Rawlings-1981-2},
Rawlings extended his bijection to words and then proved that {\footnotesize INV} and $r\text{-\footnotesize MAJ}$ are equidistributed on $\mathfrak{S}_{M}$.
Here we use a  different way to describe Rawlings' bijection $R$ on words given in \cite{Rawlings-1981-2} that takes {\footnotesize INV} to $r$-{\footnotesize MAJ}.

Given $w\in\mathfrak{S}_{M}$,
let $w^{\prime}\in\mathfrak{S}_{M^{\prime}}$, where $M^{\prime}=\{1^{k_{1}},2^{k_{2}},\ldots,(m-1)^{k_{m-1}}\}$,
be the word obtained from $w$ by deleting all of the occurrences of $m$.
Consider the $j$th $m$ in $w$ from left to right,
let $u_{j}(m)$ be the letters in $w$ that are on the right of this $m$ and that are smaller than $m$.
Clearly, $u_{1}(m)\geq u_{2}(m)\geq\cdots\geq u_{k_{m}}(m)$.
For example,
$w=2152431552$,
$w^{\prime}=2124312$,
and $u_{1}(5)=5$, $u_{2}(5)=1$, $u_{3}(5)=1$.
We will recursively define Rawlings' bijection $R$.
Assume that $R(w^{\prime})$ has been defined,
we will define $R(w)\in\mathfrak{S}_{M}$ from $R(w^{\prime})$ by inserting $k_{m}$ $m$'s.
Assume that we have inserted $(j-1)$ $m$'s,
we are going to insert the $j$th $m$.
First, we star the positions before each $m$.
Then, we label the positions that are not occupied by any star according to the following rules.
Using the labels $0,1,2,\ldots$ in order,
first read from right to left and label those positions that
will \emph{not} result in the creation of a new $r$-descent.
Then,
reading back from left to right the positions that \emph{will} create a new $r$-descent are labeled.
Finally, we insert the $j$th $m$ into the position labeled by $u_{j}(m)$.

For example,
assume that $r=3$, and we will insert the third $5$ into $215243152$,
then the labels in the top and bottom rows of
$$\includegraphics[width=6.0cm]{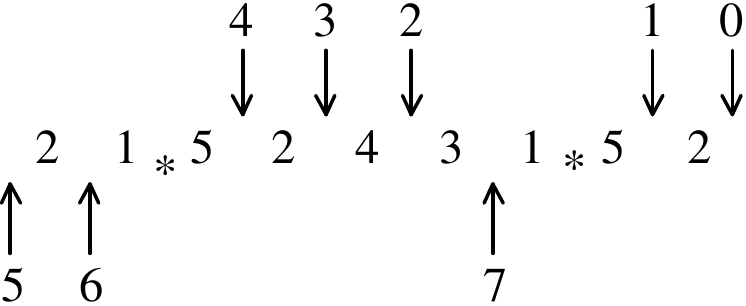}$$
respectively indicate the positions  that will not  and that will result in a new $3$-descent.
If $u_{3}(5)=1$,
we obtain $2152431552.$

With the above notations,
we now consider $w\in\mathcal{P}_{M}$.
It is not hard to see that
$w^{\prime}\in\mathcal{P}_{M^{\prime}}$ and $u_{k_{m}}(m)=0$.
By the procedure of creating $R(w)$ from $R(w^{\prime})$,
we see that the rightmost letter of  $R(w)$ is $m$.
Then an inductive proof shows that $R(w)\in\mathcal{P}_{M}$.
This is the content of the following proposition,
which implies the equidistribution of {\footnotesize INV} and $r$-{\footnotesize MAJ}  on $\mathcal{P}_{M}$.
\begin{proposition}\label{Prop-Rawlings-closed-on-M}
Let $M=\{1^{k_{1}},2^{k_{2}},\ldots,m^{k_{m}}\}$ with $k_{i}\geq1$ for all $i\in[m]$,
the set $\mathcal{P}_{M}$ is invariant under Rawlings' bijection, that is,
$$R(\mathcal{P}_{M})=\mathcal{P}_{M}.$$
\end{proposition}

\section{Equidistribution of (des,{\large MAJ}) and (exc,{\large DEN})  on $\mathcal{P}_{M}$}\label{section-bi-2}
Denert's permutation statistic, {\footnotesize DEN}, was introduced by Denert in \cite{Denert-1990},
and she conjectured that (exc, {\footnotesize DEN}) is Euler-Mahonian.
This conjecture was first proved by Foata and  Zeilberger \cite{Foata-1990},
Han \cite{Han-1990-1,Han-1990-2} gave two bijective proofs.
Han \cite{Han-1994} extended  {\footnotesize DEN} to words,
and
proved that  (exc, {\footnotesize DEN}) is Euler-Mahonian by giving a bijection on words that takes (des, {\footnotesize MAJ}) to (exc, {\footnotesize DEN}).
We denote Han's bijection given in \cite{Han-1994} by $H_{\text{\tiny DEN}}$.
The goal of this section is to establish the equidistribution of the bi-statistics (des, {\footnotesize MAJ}) and (exc, {\footnotesize DEN}) on $\mathcal{P}_{M}$
by proving that $H_{\text{\tiny DEN}}$ preserves the increasing tail permutation.

Let $w=w_{1}w_{2}\ldots w_{n}\in\mathfrak{S}_{M}$,
the two-line notation of $w$ is written as
\begin{equation*}
 w=
 \left(
 \begin{matrix}
   a_{1} & a_{2}  & a_{3}&  \ldots&a_{n}\\
   w_{1} & w_{2}  & w_{3}&  \ldots&w_{n}
  \end{matrix}
  \right),
\end{equation*}
where $\overline{w}:=a_{1}a_{2}\ldots a_{n}$ is the nondecreasing rearrangement of $w=w_{1}w_{2}\ldots w_{n}$, this notation will be adhered to throughout;
that is, if $w=w_{1}w_{2}\ldots w_{n}$ is a word,
$\overline{w}$ has the above meaning.

Let $w=w_{1}w_{2}\ldots w_{n}\in\mathfrak{S}_{M}$,
with $\overline{w}=a_{1}a_{2}\ldots a_{n}$.
An \emph{excedance} in $w$ is a triple $(i,a_{i},w_{i})$ such that $w_{i}>a_{i}$.
Here $i$ is called the \emph{excedance place},
$a_{i}$ is called the \emph{excedance bottom},
$w_{i}$ is called the \emph{excedance top}.
The number of excedances of $w$ is denoted by exc$(w)$.
Let Exc$~w$ be  the subword consisting of all the excedance tops of $w$, in the order induced by $w$.
Let Nexc$~w$ be  the subword consisting of those letters of $w$ that are not excedance tops.
For example,
if $w=121442314$,
then
Exc$~w=244$ and Nexc$~w =112314$.
Let imv$(w)$ be the number of weak inversions of $w$, i.e.,
$$\text{imv}(w)=|\{(i,j):i<j,w_{i}\geq w_{j}\}|.$$
Denert's statistic of $w$, {\footnotesize DEN}$(w)$, is defined by
$$\text{\footnotesize DEN}(w)=\sum_{i}\{i:w_{i}>a_{i}\}+\text{imv(Exc~}w)+\text{\text{\footnotesize INV}(Nexc~}w).$$
For example, let
\begin{equation*}
w= \left(
\begin{matrix}
1&1&2&2&3&3&3&4&4&5\\
5&3&1&1&2&4&4&3&2&3
\end{matrix}
  \right),
\end{equation*}
then Exc$~w=5344$ and Nexc$~w=112323$, and
$$\text{\footnotesize DEN}(w)=(1+2+6+7)+\text{imv}(5344)+\text{\footnotesize INV}(112323)=16+4+1=21.$$

Before stating Han's bijection $H_{\text{\tiny DEN}}$, we need some notions.
A \emph{biword} is an ordered pair of words of the same length,
written as
\begin{equation*}
 w=
 \left(
 \begin{matrix}
   x_{1} & x_{2}  & x_{3}&  \ldots&x_{n}\\
   w_{1} & w_{2}  & w_{3}&  \ldots&w_{n}
  \end{matrix}
  \right).
\end{equation*}
In particular, the two-line notation of a word is a  biword satisfying $x_{1}x_{2}\ldots x_{n}$ is the nondecreasing rearrangement of $w_{1}w_{2}\ldots w_{n}$.

\begin{definition} \emph{A biword
$
 w=
 \left(
 \begin{matrix}
   x_{1} & x_{2}  & x_{3}&  \ldots&x_{n}\\
   w_{1} & w_{2}  & w_{3}&  \ldots&w_{n}
  \end{matrix}
  \right)
$
is called a \emph{dominated cycle} if $n=1$ and $w_{1}=x_{n}$, or $n>1$, $w_{1}=x_{n}$, $w_{i}=x_{i-1}$ and $w_{1}>w_{i}$ for all $2\leq i\leq n$.}
\end{definition}

\begin{definition}
\emph{Let $\mathbb{P}$ be the set of positive integers.
Given $x,y\in\mathbb{P}$, the cyclic interval $\rrbracket x, y \rrbracket$ is defined by
\begin{align*}
\rrbracket x, y \rrbracket=
\begin{cases}
\{z\in \mathbb{P}:x<z\leq y\},& \text{if~}x\leq y;\\
\{z\in \mathbb{P}:x<z\text{~or~} z\leq y\}& \text{if~}x> y.
\end{cases}
\end{align*}}
\end{definition}

\begin{definition}
\emph{For $1\leq i\leq n-1$, define the operator $T_{i}$ on biword $
 w=
 \left(
 \begin{matrix}
   x_{1} & x_{2}  & x_{3}&  \ldots&x_{n}\\
   w_{1} & w_{2}  & w_{3}&  \ldots&w_{n}
  \end{matrix}
  \right)
$ to be
$$ T_{i}(w)=
 \left(
 \begin{matrix}
   x_{1} & x_{2}& \ldots&x_{i-1}\\
   w_{1} & w_{2}& \ldots&w_{i-1}
  \end{matrix}
  \right)
  T
  \left(
 \begin{matrix}
   x_{i} & x_{i+1} \\
   w_{i} & w_{i+1}
  \end{matrix}
  \right)
  \left(
 \begin{matrix}
   x_{i+2} &  \ldots&x_{n}\\
   w_{i+2} &  \ldots&w_{n}
  \end{matrix}
  \right),$$
  where
$$T
  \left(
 \begin{matrix}
   x & y \\
   \alpha & \beta
  \end{matrix}
  \right) =
\begin{cases}
  \left(
 \begin{matrix}
   y & x \\
   \beta &\alpha
  \end{matrix}
  \right),& \text{if~exactly~one~of~}\alpha\text{~and~}\beta\text{~lies~in~}\rrbracket x, y \rrbracket;\\
  \left(
 \begin{matrix}
   y & x \\
\alpha&\beta
  \end{matrix}
  \right),& \text{otherwise}.
\end{cases}
  $$}
\end{definition}

Han's bijection $H_{\text{\tiny DEN}}$ is devised to decompose a word into dominated cycles.
Below is a description of Han's bijection.

Let $w=w_{1}w_{2}\ldots w_{n}\in\mathfrak{S}_{M}$,
we write it in the two-line notation
\begin{equation*}
 w=
 \left(
 \begin{matrix}
   a_{1} & a_{2}  & a_{3}&  \ldots&a_{n}\\
   w_{1} & w_{2}  & w_{3}&  \ldots&w_{n}
  \end{matrix}
  \right).
\end{equation*}
If $n=1$, then $w$ itself is a dominated cycle.
We assume that $n\geq2$.
If $w_{n}=a_{n}$, then set
\begin{equation*}
 w^{\prime}=
 \left(
 \begin{matrix}
   a_{1} & a_{2}  & a_{3}&  \ldots&a_{n-1}\\
   w_{1} & w_{2}  & w_{3}&  \ldots&w_{n-1}
  \end{matrix}
  \right),
 ~~u=
 \left(
 \begin{matrix}
   a_{n} \\
   w_{n}
  \end{matrix}
  \right).
\end{equation*}
If $w_{n}\neq a_{n}$, let $i_{1}$ be the largest index such that $a_{i_{1}}= w_{n}$, and set
\begin{equation*}
 w^{(1)}=T_{n-2}\circ T_{n-1}\circ\cdots \circ T_{i_{1}}(w)=
 \left(
 \begin{matrix}
   a_{1}^{(1)} & a_{2}^{(1)}  &\ldots&a_{n-2}^{(1)}& w_{n}&a_{n}\\
   w_{1}^{(1)} & w_{2}^{(1)}  & \ldots&w_{n-2}^{(1)}&w_{n-1}^{(1)} &w_{n}
  \end{matrix}
  \right).
\end{equation*}
If $w_{n-1}^{(1)}=a_{n}$, then set
\begin{equation*}
 w^{\prime}=
 \left(
 \begin{matrix}
   a_{1}^{(1)} & a_{2}^{(1)}  &\ldots&a_{n-2}^{(1)}\\
   w_{1}^{(1)} & w_{2}^{(1)}  & \ldots&w_{n-2}^{(1)}
  \end{matrix}
  \right),
 ~~u=
 \left(
 \begin{matrix}
   w_{n}&a_{n} \\
   a_{n}&w_{n}
  \end{matrix}
  \right).
\end{equation*}
If $w_{n-1}^{(1)}\neq a_{n}$, let $i_{2}$ be the largest index such that $a_{i_{2}}^{(1)}= w_{n-1}^{(1)}$, and set
\begin{equation*}
 w^{(2)}=T_{n-3}\circ T_{n-1}\circ\cdots\circ T_{i_{2}}(w^{(1)})=
 \left(
 \begin{matrix}
a_{1}^{(2)} & a_{2}^{(2)}&\ldots&a_{n-3}^{(2)}& w_{n-1}^{(1)}&w_{n}&a_{n}\\
w_{1}^{(2)} & w_{2}^{(2)}& \ldots&w_{n-3}^{(2)}&w_{n-2}^{(2)}&w_{n-1}^{(1)} &w_{n}
  \end{matrix}
  \right).
\end{equation*}
Similarly, we can repeat the above process by considering whether $w_{n-2}^{(2)}$ is equal to $a_{n}$.
So we can obtain a sequence of words $w^{(1)},w^{(2)},\ldots,w^{(t)}$  such that
\begin{equation*}
 w^{(t)}=
 \left(
 \begin{matrix}
a_{1}^{(t)} & a_{2}^{(t)}&\ldots &a_{n-t-1}^{(t)}&w_{n-t+1}^{(t-1)}&\ldots&w_{n}         &a_{n}\\
w_{1}^{(t)} & w_{2}^{(t)}& \ldots&w_{n-t-1}^{(t)}&w_{n-t}^{(t)}&\ldots&w_{n-1}^{(1)} &w_{n}
  \end{matrix}
  \right),
\end{equation*}
where $w_{n-t}^{(t)}=a_{n}$ and $w_{n-t+1}^{(t-1)}\neq a_{n},\ldots,w_{n-1}^{(1)}\neq a_{n},w_{n}\neq a_{n}$.
Set
\begin{equation*}
 w^{\prime}=
 \left(
 \begin{matrix}
a_{1}^{(t)} & a_{2}^{(t)}&\ldots &a_{n-t-1}^{(t)}\\
w_{1}^{(t)} & w_{2}^{(t)}& \ldots&w_{n-t-1}^{(t)}
  \end{matrix}
  \right),
 ~~u=
 \left(
 \begin{matrix}
w_{n-t+1}^{(t-1)}&w_{n-t+2}^{(t-2)}&\ldots&w_{n}     &a_{n}\\
a_{n}    &w_{n-t+1}^{(t-1)}&\ldots&w_{n-1}^{(1)} &w_{n}
  \end{matrix}
  \right),
\end{equation*}
Note that $a_{n}\geq a_{i}$ for $1\leq i\leq n$, and
$w_{n-t+1}^{(t-1)}\neq a_{n},\ldots,w_{n-1}^{(1)}\neq a_{n},w_{n}\neq a_{n}$,
then we see that $u$ is a dominated cycle.

It is not hard to see that $w^{\prime}$ is the two-line notation of the word
$w_{1}^{(t)}w_{2}^{(t)}\ldots w_{n-t-1}^{(t)}$.
Appealing to induction, assume that $w^{\prime}$ admits a decomposition $u_{1},u_{2},\ldots,u_{s}$ into dominated cycles,
then the  decomposition of $w$ is defined to be $u_{1},u_{2},\ldots,u_{s},u$,
and $H_{\text{\tiny DEN}}(w)$ is obtained by  concatenating the bottom rows of these cycles.

\begin{theorem}[Han \cite{Han-1994}]\label{Han}
$H_{\emph{\tiny DEN}}:\mathfrak{S}_{M}\rightarrow\mathfrak{S}_{M}$ is a bijection
satisfying
$$\left(\emph{exc},\emph{\footnotesize{DEN}}\right)w=
\left(\emph{des},\emph{\footnotesize{MAJ}}\right)H_{\emph{\tiny DEN}}(w)$$ for all $w\in\mathfrak{S}_{M}$.
\end{theorem}

\begin{example}
\emph{Let $w=124324$, then
\begin{align*}
 \left(
 \begin{matrix}
1&2&2&3&4&4\\
1&2&4&3&2&4
  \end{matrix}
  \right)
&\longrightarrow
   \left(
 \begin{matrix}
1&2&\textbf{2}&3&4\\
1&2&4&3&2
  \end{matrix}
  \right)
     \left(
 \begin{matrix}
4\\
4
  \end{matrix}
  \right)
  \stackrel{T_{3}}{\longrightarrow}
    \left(
 \begin{matrix}
1&2&3&2&4\\
1&2&3&4&2
  \end{matrix}
  \right)
     \left(
 \begin{matrix}
4\\
4
  \end{matrix}
  \right)\\
 &\longrightarrow
      \left(
 \begin{matrix}
1\\
1
  \end{matrix}
  \right)
      \left(
 \begin{matrix}
2\\
2
  \end{matrix}
  \right)
       \left(
 \begin{matrix}
3\\
3
  \end{matrix}
  \right)
       \left(
   \begin{matrix}
2&4\\
4&2
  \end{matrix}
  \right)
     \left(
 \begin{matrix}
4\\
4
  \end{matrix}
  \right).
  \end{align*}
  So $H_{\text{\tiny DEN}}(w)=123424$.
  We see that (exc, {\footnotesize DEN}$)w=$
  (des, {\footnotesize MAJ}$)H_{\text{\tiny DEN}}(w)=(1,4)$.}
\end{example}

The following proposition shows that Han's bijection $H_{\text{\tiny DEN}}$ preserves the increasing tail permutation,
which implies the equidistribution of the bi-statistics (des, {\footnotesize MAJ}) and (exc, {\footnotesize DEN}) on $\mathcal{P}_{M}$.
\begin{proposition}\label{Prop-Han-closed-on-M}
Let $M=\{1^{k_{1}},2^{k_{2}},\ldots,m^{k_{m}}\}$ with $k_{i}\geq1$ for all $i\in[m]$,
the set $\mathcal{P}_{M}$ is invariant under Han's bijection $H_{\emph{\tiny DEN}}$, that is,
$$H_{\emph{\tiny DEN}}(\mathcal{P}_{M})=\mathcal{P}_{M}.$$
\end{proposition}
\begin{proof}
Given $w\in\mathcal{P}_{M}$,
assume that the operations that we used  to create $H_{\text{\tiny DEN}}(w)$
are $T_{k_{1}}$, $T_{k_{2}}$, $\ldots,$ $T_{k_{s}}$ in order.
Let
$${\alpha_{1} \choose \beta_{1}}={\overline{w} \choose w}\text{~and~}
{\alpha_{i+1} \choose \beta_{i+1}}=T_{k_{i}}{\alpha_{i} \choose \beta_{i}}\text{~for~}1\leq i\leq s.$$
Consider the sequence $$\beta_{1},\beta_{2},\beta_{3},\ldots,\beta_{s+1},$$
where $\beta_{1}=w$, $\beta_{s+1}=H_{\text{\tiny DEN}}(w)$.
It is clear that for $1\leq i\leq s$,
either $\beta_{i+1}=\beta_{i}$, or $\beta_{i+1}$ is obtained from $\beta_{i}$ by interchanging two consecutive entries.
Note that the tail permutation of $\beta_{1}$ is $123\ldots m$.
If all the tail permutations of $\beta_{1},\beta_{2},\beta_{3},\ldots,\beta_{s+1}$ are $123\ldots m$, our proof is completed.
Now we assume that all the  tail permutations of  $\beta_{1},\beta_{2},\ldots,\beta_{t}$ are $123\ldots m$,
and the  tail permutation  of $\beta_{t+1}$ is not $123\ldots m$.
Assume that $\beta_{t}=b_{1}b_{2}\ldots b_{n}$ and $\beta_{t+1}=b_{1}b_{2}\ldots b_{c-1}b_{c+1}b_{c}b_{c+2}\ldots b_{n}$.
Because the tail permutations of $\beta_{t}$ and $\beta_{t+1}$  are different and
the tail permutation of $\beta_{t}$ is $123\ldots m$,
we see that the tail permutation of $\beta_{t+1}$  has the form $12\ldots (i-1)(i+1)i(i+2)\ldots m$.
Thus, in $\beta_{t}$,
$b_{c}$ is the last occurrence of the letter $i$ and $b_{c+1}$ is the last occurrence of the letter $i+1$.
Since the tail permutation of $\beta_{t}$ is $123\ldots m$ and $b_{c}$ is the last occurrence of the letter $i$, we have $\min\{b_{c+1},b_{c+2},\ldots,b_{n}\}>i$.
Note that
\begin{equation*}\label{Eq-Han}
  \left(
 \begin{matrix}
  \alpha_{t+1} \\
   \beta_{t+1}
  \end{matrix}
  \right)=T_{k_{t}}  \left(
 \begin{matrix}
  \alpha_{t} \\
   \beta_{t}
  \end{matrix}
  \right)=
 \left(
 \begin{matrix}
   \ast & \ast& \ldots&\ast\\
   b_{1} & b_{2}& \ldots&b_{c-1}
  \end{matrix}
  \right)
  T
  \left(
 \begin{matrix}
   x & y \\
   i & i+1
  \end{matrix}
  \right)
  \left(
 \begin{matrix}
   \ast &  \ldots&\ast\\
   b_{c+2} &  \ldots&b_{n}
  \end{matrix}
  \right).
  \end{equation*}
By the  process of creating $H_{\text{\tiny DEN}}(w)$,
we can see the following two facts:
(i) $x\leq y$;
(ii) $x\in\{b_{c+2},b_{c+3},\ldots,b_{n}\}$.
From fact (i) we see that $$\rrbracket x, y \rrbracket=\{x+1,x+2,\ldots,y\}.$$
By fact (ii) we have
$$x\geq\min\{b_{c+2},b_{c+3},\ldots,b_{n}\}>i.$$
Thus, $i<i+1<x+1$.
Then  both $i$ and $i+1$ are not in the set $\rrbracket x, y \rrbracket$.
By the definition of the operator $T$, we have $\beta_{t+1}=\beta_{t}$,
which is a contradiction
and the proof is completed.
\end{proof}

\section{Equidistribution of (des,{\large MAK},{\large MAD}) and (exc,{\large DEN},{\large INV})  on $\mathcal{P}_{M}$}\label{section-tri}
The Mahonian permutation statistic {\footnotesize MAK} was introduced by Foata and Zeilberger \cite{Foata-1990}.
Clarke, Steingr\'{\i}msson and Zeng \cite{Clarke-1997} extended it to words.
In the same paper, Clarke, Steingr\'{\i}msson and Zeng introduced a new Mahonian statistic  {\footnotesize MAD} on words,
and they proved that the triple statistics (des, {\footnotesize MAK}, {\footnotesize MAD}) and (exc, {\footnotesize DEN}, {\footnotesize INV}) are equidistributed on words by exhibiting a bijection $\Phi$ on words that takes (des, {\footnotesize MAK}, {\footnotesize MAD}) to (exc, {\footnotesize DEN}, {\footnotesize INV}).
The goal of this section is to establish the equidistribution of  (des, {\footnotesize MAK}, {\footnotesize MAD}) and (exc, {\footnotesize DEN}, {\footnotesize INV}) on $\mathcal{P}_{M}$
by proving that $\Phi$  preserves the increasing tail permutation.

Let $w=w_{1}w_{2}\ldots w_{n}\in\mathfrak{S}_{M}$,
the \emph{height} $h(a)$ of a letter $a$ in $w$ is one more than the number of letters in $w$ that are strictly smaller than $a$.
The \emph{value} of the $i$th letter in $w$, denoted by $v_{i}$,
is defined by
$$v_{i}=h(w_{i})+l(i),$$
where $l(i)$ is the number of letters in $w$ that are to the left of $w_{i}$ and
equal to $w_{i}$.
For example, given $w=21144231$,
then $\overline{w}=11122344$,
so the heights of $1,2,3,4$ are, respectively, $1,4,6,7$.
The values of the letters of $w$ are given by $4,1,2,7,8,5,6,3,$ in the order in which they appear in $w$.
It is not hard to see that $v_{1}v_{2}\ldots v_{n}=\text{std}(w)$.

Let $w=w_{1}w_{2}\ldots w_{n}\in\mathfrak{S}_{M}$,
recall that a decent of $w$ is an index $i$ such that $w_{i}>w_{i+1}$,
we call $w_{i}$ a \emph{descent top},
and  $w_{i+1}$ a \emph{descent bottom}.
The \emph{descent tops sum} of $w$, denoted by $\text{Dtop}(w)$,
is the sum of the heights of the descent tops of $w$.
The \emph{descent bottoms sum} of a word $w$,
denoted by $\text{Dbot}(w)$, is the sum of the values of the descent bottoms of $w$.
The \emph{descent difference} of $w$ is
$$\text{Ddif}(w)=\text{Dtop}(w)-\text{Dbot}(w).$$

Given a word $w=w_{1}w_{2}\ldots w_{n}$,
we separate $w$ into its \emph{descent blocks} by putting in dashes between $w_{i}$ and $w_{i+1}$ whenever $w_{i}\leq w_{i+1}$.
A maximal  contiguous subword of $w$ which lies between two dashes is a \emph{descent block}.
A descent block is an \emph{outsider} if it has only one letter;
otherwise, it is a \emph{proper} descent block.
The leftmost letter of a proper descent block is its \emph{closer} and the rightmost letter is its \emph{opener}.
Let $B$ be a proper descent block of the word $w$ and let $C(B)$ and $O(B)$
be the closer and opener of $B$, respectively.
Let $a$ be a letter of $w$,
we say that $a$ is \emph{embraced} by $B$ if $C(B)\geq a> O(B)$.

The \emph{right embracing numbers} of a word $w=w_{1}w_{2}\ldots w_{n}$ are the numbers $e_{1},e_{2},\ldots,e_{n}$ where $e_{i}$ is the number of descent blocks in $w$ that are strictly to the right of $w_{i}$ and that embrace $w_{i}$.
The \emph{right embracing sum} of $w$, denoted by Res$(w)$, is defined by
$$\text{Res}(w)=e_{1}+e_{2}+\cdots+ e_{n}.$$
\begin{definition}
\begin{align*}
\emph{\footnotesize MAK}(w)&=\emph{Dbot}(w)+\emph{Res}(w),\\
\emph{\footnotesize MAD}(w)&=\emph{Ddif}(w)+\emph{Res}(w).
\end{align*}
\end{definition}

We now give an overview of the bijection $\Phi$, see \cite{Clarke-1997}.
Given $w=w_{1}w_{2}\ldots w_{n}\in\mathfrak{S}_{M}$,
let $\pi=\text{std}(w)$.
For permutation $\pi$,
we first construct two biwords
$\left(
 \begin{matrix}
   f\\
   f^{\prime}
  \end{matrix}
  \right)$
and
$
 \left(
 \begin{matrix}
   g\\
   g^{\prime}
  \end{matrix}
  \right)
$,
and then form the  biword
$
\left(
 \begin{matrix}
   f&g\\
   f^{\prime}&g^{\prime}
  \end{matrix}
  \right)
$
by concatenating $f$ and $g$, and $f^{\prime}$ and $g^{\prime}$, respectively.
The word $f$ is defined as the subword of descent bottoms in $\pi$, ordered increasingly.
The word $g$ is defined as the subword of non-descent bottoms in $\pi$, also ordered increasingly.
The word $f^{\prime}$ is the subword of descent tops in $\pi$,
ordered so that for any letter $x$ in $f^{\prime}$,
there are exactly $d$ letters in $f^{\prime}$ that are on the left of $x$ and that are greater than $x$,
where $d$ is the embracing number of the letter $x$ in $\pi$.
The word $g^{\prime}$ is the subword of non-descent tops in $\pi$,
ordered so that for any letter $x$ in $g^{\prime}$,
there are exactly $d$ letters in $g^{\prime}$ that are on the right of $x$ and that are smaller than $x$,
where $d$ is the embracing number of the letter $x$ in $\pi$.
Rearranging the columns of $
 \left(
 \begin{matrix}
   f&g\\
   f^{\prime}&g^{\prime}
  \end{matrix}
  \right)
$,
so that the top row is in increasing order,
then let $\pi^{\prime}$ be the bottom row  of the rearranged biword.
We point out that $f$ and $f^{{\prime}}$ are the excedance bottoms and excedance tops in $\pi^{\prime}$, respectively.
Finally, we let $\Phi(w)=\text{istd}_{M}(\pi^{\prime})$.

\begin{example}
\emph{Consider the word
$$w=1~3~2~1~3~2~2~3.$$
Then
$$\pi=\text{std}(w)=1-6~3~2-7~4-5-8.$$
It is not hard to see that
$$
 \left(
 \begin{matrix}
   f\\
   f^{\prime}
  \end{matrix}
  \right)=
  \left(
 \begin{matrix}
   2&3&4\\
   3&7&6
  \end{matrix}
  \right),~~~~
 \left(
 \begin{matrix}
   g\\
   g^{\prime}
  \end{matrix}
  \right)=
  \left(
 \begin{matrix}
   1&5&6&7&8\\
   1&2&4&5&8
  \end{matrix}
  \right).
$$
Then$$
 \left(
 \begin{matrix}
   f&g\\
   f^{\prime}&g^{\prime}
  \end{matrix}
  \right)=
  \left(
 \begin{matrix}
   2&3&4&1&5&6&7&8\\
   3&7&6&1&2&4&5&8
  \end{matrix}
  \right)\rightarrow
  \left(
 \begin{matrix}
   1&2&3&4&5&6&7&8\\
   1&3&7&6&2&4&5&8
  \end{matrix}
  \right),$$
  and
  $$\pi^{\prime}=1~3~7~6~2~4~5~8.$$
Thus,
$$\Phi(w)=1~2~3~3~1~2~2~3.$$}
\end{example}

\begin{theorem}[Clarke-Steingr\'{\i}msson-Zeng \cite{Clarke-1997}]\label{Clarke-1997}
$\Phi:\mathfrak{S}_{M}\rightarrow\mathfrak{S}_{M}$ is a bijection
satisfying
$$\left(\emph{des},\emph{\footnotesize{MAK}},\emph{\footnotesize{MAD}}\right)w=
\left(\emph{exc},\emph{\footnotesize{DEN}},\emph{\footnotesize{INV}}\right)\Phi(w)$$ for all $w\in\mathfrak{S}_{M}$.
\end{theorem}

The following proposition shows that $\Phi$ preserves the increasing tail permutation,
it implies that the triple statistics (des, {\footnotesize MAK}, {\footnotesize MAD}) and (exc, {\footnotesize DEN}, {\footnotesize INV}) are equidistributed on $\mathcal{P}_{M}$.
\begin{proposition}\label{Prop-Clarke-closed-on-M}
Let $M=\{1^{k_{1}},2^{k_{2}},\ldots,m^{k_{m}}\}$ with $k_{i}\geq1$ for all $i\in[m]$,
the set $\mathcal{P}_{M}$ is invariant under $\Phi$, that is,
$$\Phi(\mathcal{P}_{M})=\mathcal{P}_{M}.$$
\end{proposition}
\begin{proof}
Let $w=w_{1}w_{2}\ldots w_{n}\in\mathcal{P}_{M}$.
For any given $s\in\{2,3,\ldots,m\}$,
assume that $w_{p}=s$ is the last occurrence  of the letter $s$ in $w$.
Let $c=\sum_{j=1}^{s}k_{j}$.
Let $\pi=\pi_{1}\pi_{2}\ldots\pi_{n}=\text{std}(w)$,
then $\pi_{p}=\sum_{j=1}^{s}k_{j}=c$.
We claim that in permutation $\pi^{\prime}$, for any $b$ with $b<c$,
the letter $b$ is on the left of the letter $c$.
Note that our claim implies the proposition.
Below we prove our claim.
Because the tail permutation of $w$ is $12\ldots m$ and
$w_{p}=s$ is the last occurrence of the letter $s$ in $w$,
we have $w_{p}<w_{j}$ for $j>p$.
Then $c=\pi_{p}<\pi_{j}$ for $j>p$.
It follows that in permutation $\pi$  the embracing number of the letter $c$ is $0$
and that $c$ is a non-descent top.
By the definition of $\pi^{\prime}$,
we see that $c\in g^{\prime}$,
and there is no  letter in $g^{\prime}$ that is smaller than $c$ and that is on the right of $c$.
If $b\in g^{\prime}$, it must be on the left of $c$ in $g^{\prime}$ since $b<c$.
In this case we see that $b$ is on the left of $c$ in $\pi^{\prime}$.
If $b\in f^{\prime}$, then it is an excedance top in $\pi^{\prime}$.
Assume that $\pi_{j}^{\prime}=b$,
then $j<b$.
Assume that $\pi_{k}^{\prime}=c$,
since $c\in g^{\prime}$ is not an excedance top,
we have $c\leq k$.
Therefore, $j<b<c\leq k$.
So in this case we also have that $b=\pi_{j}^{\prime}$ is on the left of $c=\pi_{k}^{\prime}$ in  $\pi^{\prime}$ as $j<k$.
Then our claim is true and we complete the proof.
\end{proof}

\noindent{\bf Acknowledgements.}
This work was supported by the National Natural Science Foundation of China (12101134).
The author appreciates the careful review, corrections and helpful suggestions to this paper made by the referees.

\phantomsection

\end{document}